\documentclass[10pt, a4paper]{article}
\usepackage{graphicx}
\usepackage{amsmath, bm, amssymb, amsfonts, amsthm, cases, extpfeil, subfigure, slashbox, mathtools,mathrsfs,here,setspace}
\usepackage{ascmac}
\usepackage{enumerate}
\usepackage[all]{xy}
\DeclareMathOperator*{\restprod}%
 {\mathchoice{\ooalign{\ensuremath{\displaystyle\prod}\crcr\ensuremath{\displaystyle\coprod}}}%
             {\ooalign{\ensuremath{\textstyle\prod}\crcr\ensuremath{\textstyle\coprod}}}%
             {\ooalign{\ensuremath{\scriptstyle\prod}\crcr\ensuremath{\scriptstyle\coprod}}}%
             {\ooalign{\ensuremath{\scriptscriptstyle\prod}\crcr\ensuremath{\scriptscriptstyle\coprod}}}%
 }

\newtheorem{thm}{Theorem}[section]
\newtheorem{conj}[thm]{Conjecture}
\newtheorem{prop}[thm]{Proposition}

\newtheorem{rem}[thm]{Remark}

\newtheorem{lem}[thm]{Lemma}

\newcommand{\Gal}{\mathop{\mathrm{Gal}}\nolimits}
\newcommand{\Coker}{\mathop{\mathrm{Coker}}\nolimits}
\newcommand{\pr}{\mathop{\mathrm{pr}}\nolimits}
\newcommand{\Norm}{\mathop{\mathrm{Norm}}\nolimits}
\newcommand{\Aut}{\mathop{\mathrm{Aut}}\nolimits}

\newcommand{\Ann}{\mathop{\mathrm{Ann}}\nolimits}
\newcommand{\nr}{\mathop{\mathrm{nr}}\nolimits}

\newcommand{\Irr}{\mathop{\mathrm{Irr}}\nolimits}

\newcommand{\Res}{\mathop{\mathrm{Res}}\nolimits}

\newcommand{\Ind}{\mathop{\mathrm{Ind}}\nolimits}

\begin{document}
\title{On non-abelian Brumer and Brumer-Stark conjectures for monomial CM-extensions}
\author{Jiro Nomura}
\date{}
\maketitle
  \begin{abstract}
Let $K/k$ be a finite Galois CM-extension of number fields whose Galois group $G$ is monomial and $S$ a finite set of places of $k$.\ Then the ``Stickelberger element'' $\theta_{K/k,S}$ is defined.\ Concerning this element,\ Andreas Nickel formulated the non-abelian Brumer and Brumer-Stark conjectures and their ``weak'' versions.\ In this paper,\  when $G$ is a monomial group,\ we prove that the weak non-abelian conjectures are reduced to the weak conjectures for abelian subextensions.\ We write $D_{4p},\ Q_{2^{n+2}}$ and $A_4$ for the dihedral group of order $4p$ for any odd prime $p$,\ the  generalized quaternion group of order $2^{n+2}$ for any natural number $n$ and the alternating group on $4$ letters respectively.\ Suppose that $G$ is isomorphic to $D_{4p}$,\ $Q_{2^{n+2}}$ or $A_4 \times \mathbb{Z}/2\mathbb{Z}$.\ Then we prove the $l$-parts of the weak non-abelian conjectures,\ where $l=2$ in the quaternion case,\ and  $l$ is an arbitrary prime which does not split in $\mathbb{Q}(\zeta_p)$ in the dihedral case and in $\mathbb{Q}(\zeta_3)$  in the alternating case.\ In particular,\ we do not exclude the $2$-part of the conjectures and do not assume that $S$ contains all finite places which ramify in $K/k$ in contrast with Nickel's formulation.\ 
  \end{abstract}
 \section{Introduction}
 Let $K/k$ be a finite abelian  Galois CM-extension of number fields with Galois group $G$ and let $S$ be a finite set of  places of $k$ which contains all infinite places.\ Then there exists a so-called`` Stickelberger element" $\theta_{K/k,S}$  attached to $K/k$ and $S$.\ Let $\mu(K)$ be the roots of unity in $K$ and $Cl(K)$ be the ideal class group of $K$.\ We also assume that $S$ contains all finite places which ramify in $K/k$.\ Then it was proven independently in \cite{DR}, \cite{Ba} and \cite{Ca} that
\[
 \Ann_{\mathbb{Z}[G]}(\mu(K)) \theta_{K/k,S} \subset \mathbb{Z}[G].\ 
\] 
where $\Ann_{\mathbb{Z}[G]}$ is the $\mathbb{Z}[G]$-annihilator of $\mu(K)$.\  Now Brumer's conjecture asserts 
\begin{conj}[Brumer's conjecture]\label{conj:conj0}
\[
 \Ann_{\mathbb{Z}[G]}(\mu(K)) \theta_{K/k,S} \subset \Ann_{\mathbb{Z}[G]} (Cl(K)).\ 
 \]
\end{conj}
\noindent
where $\Ann_{\mathbb{Z}[G]} (Cl(K))$ is the $\mathbb{Z}[G]$-annihilator ideal of $Cl(K)$.\  In the case $k = \mathbb{Q}$,\ the field of rational numbers,\ conjecture \ref{conj:conj0} is Stickelberger's classical theorem.\ 
There exists a more refined conjecture which is called the Brumer-Stark conjecture.\ The assertion is 
\clearpage 
\begin{conj}[The Brumer-Stark conjecture]\label{conj:conj01}\ 
\begin{itemize}
 \item For any fractional ideal $\mathfrak{A}$ of $K$,\ there exists an {\it anti-unit} $\alpha \in K^{*}$ such that $\mathfrak{A}^{w_K \theta_{K/k,S}} = (\alpha)$ 
 \item $K(\alpha^{1/w_K}) / k$ is abelian.\ 
\end{itemize}
where $w_K = |\mu(K)|$,\ $K^{*}=K \setminus \{0\}$ and the term  {\it anti-unit} means that $\alpha^{1+j} = 1$ for the unique complex conjugation  $j$ in $G$.\
\end{conj}
In general,\ the Brumer-Stark conjecture implies Brumer's conjecture.\ There exists a large body of evidence of the above two conjectures; C. Greither\cite{Gr07} showed that the $p$-part of Brumer's conjecture is valid with the assumption that the $p$-part of the group of roots of unity $\mu(K)$ is a cohomologically trivial $G$-module and the appropriate equivariant Tamagawa number conjecture is valid.\ A.\ Nickel\cite{Ni10b} showed that the $p$-part of the Brumer-Stark conjecture is true if $p$ is ``non-exceptional" (also see \cite{Nib}) and the Iwasawa $\mu$-invariant vanishes (he actually proved in loc. cit. the ``the strong Brumer-Stark property",\  and this implies Brumer's conjecture and the Brumer-Stark conjecture).\ Recently,\ C.\ Greither and C. Popescu \cite{GP} showed that the $p$-part of the Brumer-Stark conjecture is true if $S$ contains all primes above $p$ and the Iwasawa $\mu$-invariant vanishes.\ There also exist  several unconditional results.\ J.\ Tate  showed that  the Brumer-Stark conjecture is true for any relative quadratic extensions \cite[\S3,\ case(c)]{Ta} and any relative extension of degree $4$ which is contained in some non-abelian Galois extension of degree $8$ \cite[\S,\ case(e)]{Ta}.\ C. Greither, X.-F. Roblot, B. Tangedal \cite{GRT04} showed that the Brumer-Stark conjecture is true for certain relative cyclic CM-extensions of degree $2p$ for any odd prime $p$.\ In $\S \ref{sec:appl}$,\ we will use the results and the methods in the last three papers in an essential way.\

In \cite{Ni11},\  Nickel formulated non-abelian generalizations of Conjecture \ref{conj:conj0} and Conjecture \ref{conj:conj01} (see \S \ref{sec:conjectures},\  also see the original paper \cite{Nia}).\ In this paper,\ we call the conjectures the non-abelian Brumer conjecture and the non-abelian Brumer-Stark conjecture.\  In \cite{Bu1},\ D. Burns independently  gave another non-abelian generalization of Brumer's conjecture which implies Nickel's version  and in \cite{Bu2} he also proved that ``the Gross-Stark conjecture" implies a refinement of the odd $p$-part of his conjecture with the assumption that the Iwasawa $\mu$-invariant vanishes if $p$ divides [$K$:$k$].\  As well as for the abelian case,\ there exists several results concerning the non-abelian conjectures;Nickel\cite{Ni11}  showed that the $p$-part of the non-abelian Brumer conjecture and the non-abelian Brumer-Stark conjecture hold if $p$ is ``non-exceptional'' and the Iwasawa $\mu$-invariant vanishes.\ Also in \cite{Nic} he generalized the method of \cite{GP} to non-abelian settings and  proved that the odd $p$-part of the non-abelian Brumer conjecture and the non-abelian Brumer-Stark conjecture hold if $S$ contains all primes above $p$ and the Iwasawa $\mu$-invariant vanishes.\ These results are proven via the non-commutative Iwasawa main conjecture for $1$-dimensional $p$-adic Lie extensions (or,\ equivalently,\  the equivariant Iwasawa main conjecture of Ritter and Weiss) which was proven independently by M. Kakde \cite{Ka} and J. Ritter and A. Weiss \cite{RW}.\ 

\par
In what follows,\ we remove the assumption $G$ is abelian.\ Let $\Lambda^{\prime}$ be a maximal $\mathbb{Z}$-order in $\mathbb{Q}[G]$ which contains $\mathbb{Z}[G]$ and $\mathfrak{F}(G)$ denote the central conductor of $\Lambda^{\prime}$ over $\mathbb{Z}[G]$.\ For any prime $\mathfrak{p}$ of $k$,\ we fix a prime $\mathfrak{P}$ of $K$ above $\mathfrak{p}$ and a Frobenius automorphism $\phi_{\mathfrak{P}}$ at $\mathfrak{P}$.\ For another finite set $T$ of places of $k$,\ we set $E_S^{T}:=\{x \in E_S \mid x \equiv 1 \mod \prod_{\mathfrak{P} \in T(K)} \mathfrak{P}\}$ where $E_S$ is the group of $S$-units of $K$.\  We refer to the following condition as $Hyp(S,T)$;
\begin{itemize}
\item $S$ contains all infinite places and all ramifying places
\item $S \cap T=\emptyset$
\item $E_S^{T}$ is torsion free  
\end{itemize}
For each finite set $T$ of places of $k$ such that $Hyp(S,T)$ is satisfied,\ we define 
\[
 \delta_T := \nr(\prod_{\mathfrak{p} \in T}1-\phi_{\mathfrak{P}}^{-1}N\mathfrak{p})
\]
and set 
\[
 \mathfrak{A}_S:=\langle \delta_T \mid \text{$Hyp(S,T)$ is satisfied} \rangle_{\zeta(\mathbb{Z}[G])}
\]
where $\nr$ is the reduced norm of the semisimple algebra $\mathbb{Q}[G]$ and $\zeta(\mathbb{Z}[G])$ is the center of $\mathbb{Z}[G]$.\ By \cite[Lemma 1.1]{Ta84},\ $\mathfrak{A}_S$ coincides with the $\mathbb{Z}[G]$-annihilator of the roots of unity in $K$ if $G$ is abelian.\  In \cite{Ni11},\ Nickel also formulated weaker versions of  non-abelian Brumer conjecture and the weak non-abelian Brumer-Stark conjecture (for the details,\ see Conjecture \ref{conj:brumer} and \ref{conj:brumer-stark} in \S 2).\ The statement of the former conjecture is 
\begin{conj}[The weak non-abelian Brumer conjecture]\label{conj:brumer0}
\ 
\begin{itemize}
 \item $\mathfrak{A}_S \theta_S \subset \zeta(\Lambda^{\prime})$,\ 
 \item For any $x \in \mathfrak{F}(\mathbb{Z}[G])$,\ $x\mathfrak{A}_S \theta_S \subset \Ann_{\mathbb{Z}[G]}(Cl(K))$.\ 
\end{itemize}
\end{conj}
\noindent
In \cite{BJ},\ D. Burns and H. Johnston showed that for any odd prime $p$,\ the ``$p$-part" of this conjecture holds with some modification of $\theta_{S}$ if $p$ is unramified in $K/\mathbb{Q}$ and every inertia subgroup is normal in $G$.\ They deduce the weak annihilation result from  ``the Strong Stark conjecture" at $p$ which holds by \cite[Corollary 2]{Ni10b} in this case.\
\par
For any $\alpha \in K^{*}$,\ we set $S_\alpha:=\{\mathfrak{p} \mid \mathfrak{p}$ is a prime of $k$ and $\mathfrak{p}$ divides $N_{K/k}\alpha$\} and $w_K := \nr(|\mu_K|)$.\ Then the statement of the latter conjecture is  
\begin{conj}[The weak non-abelian Brumer-Stark conjecture]\label{conj:brumer-stark0}
\ 
\begin{itemize}
 \item $w_K \theta_{K/k,S} \in \zeta(\Lambda^{\prime})$,\ 
 \item For any fractional ideal $\mathfrak{A}$ of $K$ and for each $x \in \mathfrak{F}(\mathbb{Z}[G])$,\ there exists anti unit $\alpha = \alpha(\mathfrak{A},S,x)$ such that $\mathfrak{A}^{xw_K \theta_{K/k,S}} = (\alpha)$,\
\end{itemize}
and for any finite set T of places of $k$ which satisfies $Hyp(S \cup S_{\alpha},T)$,\ there exists $\alpha_T \in E_{S_\alpha}^{T}(K)$ such that 
\begin{equation*}
 \alpha^{z \delta_T} = \alpha_T^{zw_K}
\end{equation*}
for each $z \in \mathfrak{F}(\mathbb{Z}[G])$.\ 
\end{conj}
\noindent 
In \cite{Nia},\ Nickel showed that for any odd prime $p$,\ the ``$p$-part" of the above two conjectures hold if no prime above $p$ splits in $K/K^{+}$ whenever $K^{cl} \subset (K^{cl})^{+}(\zeta_p)$,\ where the superscripts $^{+}$ and $^{cl}$ mean the maximal real subfield and  the Galois closure over $\mathbb{Q}$ respectively and $\zeta_p$ is a primitive complex $p$-th root of unity.\ This result is an improvement of the result in \cite{BJ} stated above and also deduced from  ``the Strong Stark conjecture" at $p$.\ 
\par
In this paper,\ we study these weak conjectures by a direct approach.\ Although the known results which we have seen are proven via the non-commutative Iwasawa main conjecture or the strong Stark conjecture,\ we attack these weak conjectures in a different way,\ that is,\ we deduce the weak non-abelian  conjectures from abelian ones.\  We recall that a finite group $G$ is called {\it monomial} if each of the irreducible characters of $G$ is induced by a $1$-dimensional character of a certain subgroup.\  
Now our main theorems are the following:
\begin{thm}\label{thm:brumer0}Let $p$ be a rational prime (not necessarily odd) and $S$ a finite set of places of $k$ which contains all infinite places.\ We assume $G$ is monomial.\ If Brumer's conjecture (resp. the $p$-part of Brumer's conjecture) is true for an explicit list of abelian subextensions of $K/k$ (for the details of this list,\ see the list (\ref{eq:list}) in \S \ref{sec:main}), the weak non-abelian Brumer conjecture  (resp. the $p$-part of the weak non-abelian Brumer conjecture) is true for $K/k$ and $S$.\ 
\end{thm}
\begin{thm}\label{thm:brumer-stark0}
The statement of Theorem \ref{thm:brumer0} holds,\ with ``Brumer'' replaced by ``Brumer-Stark'' throughout.\  
\end{thm}
\noindent
Concerning the set $S$ of places, we stress here that $S$ does not have to contain places which ramify in $K/k$.\ Hence,\ if we assume Brumer's conjecture and the Brumer-Stark conjecture for abelian extensions,\ we get stronger annihilation results than Nickel's formulation.\ As far as the author knows,\ this observation on the set $S$ has not been made before.\ The possibility to remove the ramifying places is revealed because of the comparison between conjectures for non-abelian extensions and those for abelian extensions.\ 
\par 
Let $p$ be an odd prime and $n$ be a non-zero natural number.\ We denote by $D_{4p}$ the dihedral group of order $4p$,\ by $Q_{2^{n+2}}$ the generalized quaternion group of order $2^{n+2}$ (so $Q_8$ is Hamilton's usual group of quaternions) and by $A_4$ the alternating  group on $4$ letters.\ These three groups are monomial.\  In \S $\ref{sec:appl}$,\ we apply the above two theorems to the case $G$ is isomorphic to $D_{4p}$,\ $Q_{2^{n+2}}$ or $\mathbb{Z}/2\mathbb{Z} \times A_4$ and prove the following three theorems (\S \ref{sec:appl},\ Theorem \ref{thm:brumer1},\ \ref{thm:brumer2} and \ref{thm:brumer3}).\ 
\begin{thm}\label{thm:di}
Let $K/k$ be a finite Galois CM-extension whose Galois group is isomorphic to $D_{4p}$ and $S$ be a finite set of places of $k$ which contains all infinite places.\  Then  

\begin{description}
\item[(1)] the $p$-part of the non-abelian Brumer and Brumer-Stark conjectures are true for $K/k$ and $S$,\ 
\item[(2)] for each prime $l$ ($l \not = p$,\ $l$ can be $2$) which does not split in $\mathbb{Q}(\zeta_p)$,\  the $l$-part of the weak non-abelian Brumer and Brumer-Stark conjectures are true for $K/k$ and $S$.\ 
\end{description}
\end{thm}
\begin{thm}\label{thm:qua}
Let $K/k$ be a finite Galois CM-extension whose Galois group is isomorphic to $Q_{2^{n+2}}$ and $S$ be a finite set of places of $k$ which contains all infinite places.\ Then the $2$-part of the weak non-abelian Brumer and Brumer-Stark conjectures are true for $K/k$ and $S$,\ 
\end{thm}
\begin{thm}\label{thm:alt}
Let $K/k$ be a finite Galois CM-extension whose Galois group is isomorphic to $\mathbb{Z}/2\mathbb{Z} \times A_4$ and $S$ be a finite set of places of $k$ which contains all infinite places.\ Then 
\begin{description}
\item[(1)] the $2$-part and the 3-part of the weak non-abelian Brumer and Brumer-Stark conjectures are true for $K/k$ and $S$,\ 
\item[(2)] for each odd prime $l$ apart from $3$ which does not split in $\mathbb{Q}(\zeta_{3})$,\  the $l$-part of the non-abelian Brumer and Brumer-Stark conjectures are true for $K/k$ and $S$.\ 
\end{description}
\end{thm}
\begin{rem}\ 
\begin{enumerate}
\item The above three theorems say that the set $S$ does not have to contain the ramifying places by contrast to the  Nickel's formulation.\ Hence,\ we give stronger results than Nickel's work in \cite{Nia}.\
\item  In $Q_{2^{n+2}}$ cases,\ $\theta_{K/k, S_{\infty} \cup S_{ram}}$ always coincides with $\theta_{K/k, S_{\infty}}$ because all the subgroups of $Q_{2^{n+2}}$ are normal and all the irreducible totally odd representations are faithful (hence the Euler factors corresponding to $S_{ram}$  are $1$).\ In $D_{4p}$ cases,\ we can find explicit examples in which $\theta_{K/k, S_{\infty} \cup S_{ram}}$ does not coincide with $\theta_{K/k,S_{\infty}}$.\ In \S \ref{sec:example},\ we give a concrete $D_{12}$ extension $K/\mathbb{Q}$ so that $\theta_{K/\mathbb{Q}, S_{\infty}}$ does not coincide with $\theta_{K/\mathbb{Q}, S_{\infty} \cup S_{ram}}$.\
\item Our results contain the $2$-part of conjectures which is excluded in  \cite{BJ} and \cite{Nia}.\ As far as the author knows,\ there are no results concerning  the $2$-part of conjectures.\ Hence our main theorems give the first evidence for the $2$-part of the conjectures.\
\item  We use only the analytic class number formula to prove the above three results  by contrast to known results which were proven via  the non-commutative Iwasawa main conjecture or the strong Stark conjecture.\
\end{enumerate}
\end{rem}
\par
I would like to express my sincere grateful to Masato Kurihara for his encouragement and helpful suggestions.\ His suggestion is the starting point of this work,\ and this work could not have been completed without his guidance.\ 
I would like to thank Andreas Nickel for his helpful comments on this paper.\ In particular,\  he indicated Lemma \ref{lem:relation3} in \S \ref{sec:conjectures} and pointed out some mistakes in a draft version of this paper.\
I would also like to thank David Burns for his helpful suggestions and comments,\ especially,\ his indication of the validity of Theorem \ref{thm:brumer3} and a generalization of this article (the author study only supersolvable extensions before his suggestion).\ 
Finally,\ I  am deeply grateful to the referee for his/her careful reading of an earlier version of this paper.\ He/She pointed out many mistakes of my English writing.\

 \section*{Notations}
  For any ring $A$,\ $\zeta(A)$ denotes the center of $A$ and $M_n(A)$ denotes the ring of matrices over $A$ for some $n \in \mathbb{N}$.\ For any number field $F$,\ $\mu(F)$ denotes the roots of unity and $Cl(F)$ denotes the ideal class group of $F$.\ 

\section{Preliminaries}
\subsection{Group rings and idempotents}\label{sec:groupalg}
Let $\mathbb{\overline{Q}}$ be an algebraic closure of $\mathbb{Q}$ and we fix an embedding $\overline{\mathbb{Q}} \hookrightarrow \mathbb{C}$.\ For any finite group $G$,\ $\Irr G$ denotes the set of all irreducible $\mathbb{C}$-valued characters of $G$.\ 
We put 
\[
e_\chi := \frac{\chi(1)}{|G|}\sum_{g \in G} \chi(g^{-1})g,\ \pr_{\chi} := \frac{|G|}{\chi(1)} e_\chi = \sum_{g \in G} \chi(g^{-1})g,\ \chi \in \Irr G.\ 
\]
Then the elements $e_\chi$ are orthogonal  central primitive idempotents of $\overline{\mathbb{Q}}[G]$ and we have the Wedderburn decomposition 
\[
 \mathbb{\overline{Q}}[G] = \bigoplus_{\chi \in \Irr G} \mathbb{\overline{Q}}[G] e_\chi \cong \bigoplus_{\chi \in \Irr G} M_{\chi(1)}(\mathbb{\overline{Q}}),\ \zeta(\mathbb{\overline{Q}}[G]) = \bigoplus_{\chi \in \Irr G} \mathbb{\overline{Q}} e_\chi \cong  \bigoplus_{\chi \in \Irr G} \mathbb{\overline{Q}}.\ 
\]
The above isomorphisms imply the isomorphism
\[
\zeta(\mathbb{Q}[G]) \cong  \bigoplus_{\chi \in \Irr G/\sim} \mathbb{Q}(\chi)
\]
where we put $\mathbb{Q}(\chi):=\mathbb{Q}(\chi(g):g \in G)$ and the direct sum runs over all irreducible characters modulo $\Gal(\mathbb{\overline{Q}}/\mathbb{Q})$-action.\ We note that any element $(\alpha_{\chi})_\chi$ in the right hand side corresponds to
\[
 \sum_{\chi \in \Irr G/\sim} \sum_{\sigma \in \Gal(\mathbb{Q}(\chi)/\mathbb{Q})} \alpha_{\chi}^{\sigma} e_{\chi^{\sigma}}
\]
in the left hand side where we put $\chi^{\sigma} := \sigma \circ \chi$.\  Let $\Lambda^{\prime}$ be a maximal $\mathbb{Z}$-order in $\mathbb{Q}[G]$ which contains $\mathbb{Z}[G]$.\ If we let $\mathfrak{o}_\chi$ denote the ring of integers of $\mathbb{Q}(\chi)$,\ $\zeta(\Lambda^{\prime})$ coincides with $\bigoplus_{\chi \in \Irr G/\sim} \mathfrak{o}_\chi$ which is the unique maximal order in $\zeta(\mathbb{Q}[G])$.\ 
For each $1$-dimensional character $\chi \in \Irr G$,\ let $\chi^{\prime}$ be the character of $G/\ker\chi$ whose inflation to $G$ is $\chi$.\ Then we easily see $e_\chi = e_{\chi^{\prime}}(1/|\ker \chi)\Norm_{\ker\chi}$ where $\Norm_{\ker\chi} := \sum_{h \in \ker\chi}h$.\ If $\chi$ is induced by a character of a subgroup of $G$,\ we can write down $e_\chi$ by the following lemma:
\begin{lem}\label{lem:idempotent}
Let $G$ be a finite group and let $H$ be a subgroup of $G$.\ If an irreducible character $\chi$ of $G$ is induced by an irreducible character of $H$, we have 
\[
e_{\chi} = \sum_{\substack{\phi \in \Irr H/\sim_{\chi},\ \\ \Ind\phi = \chi}} \sum_{h \in \Gal(\mathbb{Q}(\phi)/\mathbb{Q}(\chi))} e_{\phi^{h}}.\ 
\]
where $\Irr H/\sim_\chi$ means $\Irr H$ modulo $\Gal(\overline{\mathbb{Q}}/\mathbb{Q}(\chi))$-action.\ 
\end{lem}
\begin{proof}
Since $\overline{\mathbb{Q}}[G]$ is a left and right $\overline{\mathbb{Q}}[H]$-algebra,\ $\overline{\mathbb{Q}}[G]$ decomposes into 
\[
\bigoplus_{\phi \in \Irr H} \mathbb{\overline{Q}}[G]e_\phi
\]
and the components $\mathbb{\overline{Q}}[G]e_\phi$ are left and right $\overline{\mathbb{Q}}[H]$-algebras.\ 
By the Frobenius reciprocity law,\ we have $\langle \chi,\ \Ind^{G}_{H}\phi \rangle_G = \langle \Res^{G}_{H}\chi,\ \phi\rangle_H$ where $\langle\  ,\ \rangle$ is the usual inner product of characters.\ This implies that the simple component $\mathbb{\overline{Q}}[G]e_\chi$ of $\mathbb{\overline{Q}}[G]$ decomposes into 
\[
 (\bigoplus_{\phi \in \Irr H} \mathbb{\overline{Q}}[G]e_\phi)e_\chi = \bigoplus_{\phi \in \Irr H,\ \Ind^{G}_{H}\phi = \chi} \mathbb{\overline{Q}}[G]e_\phi 
\]
as a left and  right $\mathbb{\overline{Q}}[H]$-algebra.\ This implies 
\begin{equation}\label{eq:idempotent1}
e_\chi = \sum_{\substack{\phi \in \Irr H,\ \\ \Ind^{G}_{H} \phi = \chi}} e_\psi.\ 
\end{equation}
Take a character $\phi \in \Irr H$ such that $\Ind^{G}_H \phi = \chi$.\ Then for each $g \in G$,\ we have
\[
 \chi(g) = \sum_{\substack{\tau \in G,\ \\ \tau^{-1}g\tau \in H}} \phi(\tau^{-1}g\tau).\ 
\]
Hence,\ we have $\Ind^{G}_H \phi^{\sigma} = \Ind^{G}_H \phi$ for all $\sigma \in \Gal(\mathbb{Q}(\phi)/\mathbb{Q}(\chi))$.\ Combining this with (\ref{eq:idempotent1}),\ we have
\[
 e_{\chi} = \sum_{\substack{\phi \in \Irr H/\sim_{\chi},\ \\ \Ind^{G}_H\phi = \chi}} \sum_{h \in \Gal(\mathbb{Q}(\phi)/\mathbb{Q}(\chi))} e_{\phi^{h}}.\ 
\]
This completes the proof.\ 
\end{proof}
Let $p$ be a prime.\ All the above arguments are valid if we replace $\mathbb{Q}$ by $\mathbb{Q}_p$ and $\mathbb{C}$ by $\mathbb{C}_p$,\ where $\mathbb{Q}_p$ denotes the $p$-adic completion of $\mathbb{Q}$ and $\mathbb{C}_p$ denotes the $p$-adic completion of a fixed algebraic closure $\overline{\mathbb{Q}}_p$ of $\mathbb{Q}_p$.\

\subsection{Reduced norms and conductors}\label{reducednorm}
Let $\mathfrak{o}$ be a Dedekind domain,\ $F$ the quotient field of $\mathfrak{o}$ and $A$ a separable semisimple algebra over  $F$.\ Then $A$ has a Wedderburn decomposition $A \cong \bigoplus_{i=1}^{t} A_i$ where $A_i$ is a finite dimensional central simple algebra over $F_i$ and $F_i$ is a finite separable extension over $F$.\ We set $s_i:=[F_i:F]$.\  Let $E$ be a splitting field for $A$ (we can choose $E$ so that $E$ is a finite Galois extension over $F$) and let $\nr_A$ denote the following composition map;
\[
\nr_A : A \rightarrow E \otimes_F A \cong \bigoplus_{i=1}^{t} M_{n_i}(E)^{\oplus s_i} \xrightarrow{{\oplus \det_E}^{\oplus s_i}} \zeta(E \otimes A).
\]
The image of $A$ actually lies in $\zeta(A)$ and does not depend on the choice of $E$.\ This map is called the reduced norm of $A$.\ We extend this map to any ring of matrices over $A$ by means of 
\[
\nr_A : M_m(A) \rightarrow M_m(E \otimes_F A) \cong \bigoplus_{i=1}^{t} M_{mn_i}(E)^{\oplus s_i} \xrightarrow{{\oplus \det_E}^{\oplus s_i}} \zeta(E \otimes A).
\]
 The image of $M_m(A)$ also lies in $\zeta(A)$ and does not depend on the choice of $E$ (for details of the reduced norm map,\ see \cite[\S9]{Re} and \cite[\S 7D]{CR}).\  Now we fix an $\mathfrak{o}$-order $\Lambda$ and a maximal $\mathfrak{o}$-order $\Lambda^{\prime}$ which contains $\Lambda$.\ Unfortunately the reduced norm of $A$ does not take $\Lambda$ to its center $\zeta(\Lambda)$ but to $\zeta(\Lambda^{\prime})$ in general.\ For this reason,\ we have to consider some conductors of $\Lambda^{\prime}$ over $\Lambda$.\ First we define 
\[
\mathfrak{F}(\Lambda) := \{ x \in \zeta(\Lambda^{\prime}) \mid x \Lambda^{\prime} \subset \Lambda\} (\subset \zeta(\Lambda)) .\ 
\]
This set is called the {\it central conductor} of $\Lambda^{\prime}$ over $\Lambda$.\ In the case $F = \mathbb{Q}$ or $\mathbb{Q}_p$,\ $A = F[G]$ and  $\Lambda = \mathfrak{o}[G]$ with finite group $G$,\ by Jacobinski's central conductor formula (\cite[Theorem 3]{Ja} also see \cite[\S27]{CR}),\  we see the explicit structure of $\mathfrak{F}(\Lambda)$ as 
\begin{equation} \label{eq:cond}
\mathfrak{F}(\Lambda) \cong \bigoplus_{\chi \in \Irr G /\sim} \frac{|G|}{\chi(1)}\mathfrak{D}^{-1}(F(\chi)/F)
\end{equation}
where $\mathfrak{D}^{-1}(F(\chi)/F)$ is the inverse different of $F(\chi) := F(\chi(g);g \in G)$ over $F$ and $\chi$ runs over all irreducible characters of $G$ modulo $\Gal(\overline{F}/F)$-action.\ We note that the element $((|G|/|\chi(1)|)\alpha_\chi)_\chi$ in the right hand side of (\ref{eq:cond}) corresponds to $\sum_{\chi \in \Irr G/\sim}\sum_{\sigma \in \Gal(F(\chi)/F)}\alpha_\chi^{\sigma} \pr_{\chi^{{\sigma}}}$ in the left hand side.\ In what follows,\ we only consider the case $A$ = $F[G]$  and $\Lambda = \mathfrak{o}[G]$.\ Next we define
\[
 \mathcal{ H}(\Lambda) := \{ x \in \zeta(\Lambda^{\prime}) \mid \text{$xH^{*} \in M_n(\Lambda)$ for any $H \in M_n(\Lambda)$ for all $n \in \mathbb{N}$}\}
\]
where $H^{*}$ is the  matrix over $\Lambda^{\prime}$ defined in \cite[\S 3.6]{JN} such that $HH^{*} = H^{*}H = \nr(H) \cdot 1_{n \times n}$.\ The matrix $H^{*}$ is a non-commutative analogue of the adjoint matrix and was first considered by Nickel in \cite{Ni10} (H. Johnston and Nickel\cite{JN}  introduce a slightly different definition of  $H^{*}$).\ Since $H^{*}$ lies in $M_n(\Lambda^{\prime})$,\ $\mathfrak{F}(\Lambda)$ is obviously contained in $\mathcal{ H}(\Lambda)$.\ $\mathcal{ H}(\Lambda)$ appears in a natural way in the context of the non-commutative Fitting invariants (cf. \cite[Theorem 4.2]{Ni10} and \cite[Theorem 1.2]{Ni11}).\ We set 
\[
\mathcal{ I}(\Lambda) := \langle \nr_A(H) \mid H \in M_n(\Lambda),\ n \in \mathbb{N}\rangle_{\zeta(\Lambda)}.\ 
\]
Then we get the following relation between $\mathfrak{F}(\Lambda)$ and $\mathcal{ H}(\Lambda)$: 
\begin{prop}[\cite{JN},\ Remark 6.5 and Corollary 6.20]\label{prop:conductor}
Let $p$ be a prime and $\Lambda$ be $\mathbb{Z}_p[G]$ with a finite group $G$.\ We assume $\mathcal{ I}(\Lambda) = \zeta(\Lambda^{\prime})$ and the degrees of all irreducible characters of $G$ are prime to $p$.\ Then we have 
\[
\mathcal{ H}(\Lambda) = \mathfrak{F}(\Lambda).\ 
\] 
\end{prop}
We conclude this section with the following lemma:
\begin{lem}\label{lem:conductor}Let $\chi$ be an irreducible character of $G$ which is induced by an irreducible character of a subgroup $H$ of $G$.\ Take an arbitrary element $x$ in $\mathfrak{F}(\Lambda)$ of the form
\[
 x = \sum_{\sigma \in \Gal(F(\chi)/F)} x_{\chi}^{\sigma} \pr_{\chi^{\sigma}}.\ 
\]
Then we have 
\[
 x = \sum_{\substack{\phi \in \Irr H/\sim,\ \\ \exists \sigma \in \Gal(F(\chi)/F),\ \Ind \phi = \chi^{\sigma}}} \sum_{g \in \Gal(F(\phi)/F)} x_{\chi}^{g} \pr_{\phi^g}.\ 
\]
In particular,\ $x$ also lies in $\mathfrak{F}(\mathfrak{o}[H])$.\ 
\end{lem}
\begin{proof}
For each $\sigma \in \Gal(F(\chi)/F)$,\ we fix an extension $\tilde{\sigma}$ to $\Gal(F(\phi)/F)$.\ Then we have
\begin{eqnarray*}
\lefteqn{\sum_{\substack{\phi \in \Irr H/\sim,\ \\ \exists \sigma \in \Gal(F(\chi)/F),\ \Ind \phi = \chi^{\sigma}}} \sum_{g \in \Gal(F(\phi)/F)} x_{\chi}^{g} \pr_{\phi^g}} \\
&=& \sum_{\substack{\phi \in \Irr H/\sim,\ \\ \exists \sigma \in \Gal(F(\chi)/F),\ \Ind \phi = \chi^{\sigma}}} \sum_{\sigma \in \Gal(F(\chi)/F)} \sum_{h \in \Gal(F(\phi)/F(\chi))} x_{\chi}^{\tilde{\sigma}h} \pr_{\phi^{\tilde{\sigma}h}} \\
&=& \sum_{\substack{\phi \in \Irr H/\sim,\ \\ \exists \sigma \in \Gal(F(\chi)/F),\ \Ind \phi = \chi^{\sigma}}} \sum_{\sigma \in \Gal(F(\chi)/F)} (\sum_{h \in \Gal(F(\phi)/F(\chi))} x_{\chi}^{h} \pr_{\phi^h})^{\tilde{\sigma}}\\
&=& \sum_{\substack{\phi \in \Irr H/\sim,\ \\ \exists \sigma \in \Gal(F(\chi)/F),\ \Ind \phi = \chi^{\sigma}}} \sum_{\sigma \in \Gal(F(\chi)/F)} (\sum_{h \in \Gal(F(\phi)/F(\chi))} x_{\chi}^{h} \pr_{\phi^h})^{\sigma} \\
&=& \sum_{\sigma \in \Gal(F(\chi)/F)} x_{\chi}^{\sigma} \sum_{\substack{\phi \in \Irr H/\sim_{\chi^{\sigma}},\ \\ \Ind \phi = \chi^{\sigma}}} (\sum_{h \in \Gal(F(\phi)/F(\chi))} \pr_{\phi^h})^{\sigma} \\
&=& \sum_{\sigma \in \Gal(F(\chi)/F)} x_{\chi}^{\sigma} (\sum_{\substack{\phi \in \Irr H/\sim_{\chi},\ \\ \Ind \phi = \chi}}\sum_{h \in \Gal(F(\phi)/F(\chi))}  \pr_{\phi^h})^{\sigma} \\
&=& \sum_{\sigma \in \Gal(F(\chi)/F)} x_{\chi}^{\sigma} \pr_{\chi^{\sigma}}.\ 
\end{eqnarray*}
The last equality follows from  Lemma \ref{lem:idempotent}.\ Since $x_{\chi}$ also lies in $D^{-1}(F(\phi)/F)$,\ $x$ lies in $\mathfrak{F}(\mathfrak{o}[H])$.\ 
\end{proof}

 \subsection{Stickelberger elements}\label{sec:stick}
 Let $K/k$ be a finite Galois extension of number fields with Galois group $G$.\ For any finite place  $\mathfrak{p}$ of $k$ we fix a finite place $\mathfrak{P}$ of $K$ above $\mathfrak{p}$ and $G_{\mathfrak{P}}$ (resp.\ $I_{\mathfrak{P}}$) denotes the decomposition subgroup (resp.\ inertia subgroup) of $G$ at $\mathfrak{P}$.\ Moreover,\ we fix a lift $\phi_{\mathfrak{P}}$ of the Frobenius automorphism  of $G_{\mathfrak{P}}/I_{\mathfrak{P}}$.\  
 \par
Let $S$ be a finite set of places of $k$ containing all infinite places of $k$ and let $T$ be another finite set of places of $k$ which are unramified in $K$ such that $S \cap T = \emptyset$. For any irreducible character $\chi$ of $G$,\ we put $\delta_T(\chi) := \prod_{\mathfrak{p} \in T} \det(1- \phi_{\mathfrak{P}}^{-1} N\mathfrak{p} \mid V_{\chi})$ and $L_{S}(s,\chi,K/k)$ denotes the $S$-imprimitive Artin $L$-function attached to $\chi$.\ Then we define 
\[
 \theta_{K/k,S}^{T} := \sum_{\chi \in \Irr G} \delta_T(\chi) L_S(0,\overline{\chi},K/k) e_{\chi} \in \zeta(\mathbb{C}[G])
\]
where $\overline{\chi}$ is the contragredient character of $\chi$.\ We call this element the $(S,T)$-modified Stickelberger element.\ When $S$ is the set of all ramifying places and all infinite places and $T$ is empty,\ we put $\theta_{K/k}:= \theta_{K/k,S}^{T}$ .\ Moreover,\ in the case $k$ = $\mathbb{Q}$ we will always omit the trivial character component of $\theta_{K/k,S}^{T}$.\ We can also express this element by 
\[
 \theta^{T}_{K/k,S} = \nr_{\mathbb{Q}[G]}(\prod_{\mathfrak{p} \in T}(1-\phi_{\mathfrak{P}}^{-1}N\mathfrak{p})) \theta_{K/k,S}.\ 
\]
The $(S,T)$-modified Stickelberger element is characterized by the formula
\begin{equation}\label{eq:eq1}
\chi(\theta_{K/k,S}^{T}) = \chi(1) \delta_T(\chi) L_S(0,\chi,K/k).\ 
\end{equation}
Now we assume $K/k$ is a CM-extension,\ that is,\ $k$ is a totally real field,\ $K$ is a CM-field and the complex conjugation induces a unique automorphism $j$ which lies in the center of $G$.\ We call a character $\chi$ odd if $\chi(j) = -\chi(1)$ and even otherwise.\ Then $L(0,\chi,K/k) = 0$ if $\chi$ is an even character.\ For an odd character $\chi$,\ we get $L(0,\chi,K/k)^{\sigma} = L(0,\chi^{\sigma},K/k)$ for all $\sigma \in \Aut(\mathbb{C})$ (proven by Siegel\cite{Si} if $G$ is abelian and the general result is given by Brauer induction\cite[Theorem 1.2]{Ta84}),\ which implies $\theta_{K/k,S}^{T}$ actually lies in $\zeta(\mathbb{Q}[G])$.\ Finally we put 
$\epsilon_{\chi,S} := \lim_{s \to 0} \prod_{\mathfrak{p} \in S \backslash S_{\infty}} \det(1- \phi_{\mathfrak{P}}N\mathfrak{p}^{-s} | V_{\chi}^{I_{\mathfrak{P}}})$ and define $\epsilon_S := \sum_{\chi \in \Irr G}\epsilon_{\chi,S} e_{\chi} $.\ Let $\Lambda^{\prime}$ be any maximal $\mathbb{Z}$-order in $\mathbb{Q}[G]$ which contains $\mathbb{Z}[G]$.\ Then we have 
\begin{lem}\label{lem:epsilon}
$\epsilon_S$ lies in $\zeta(\Lambda^{\prime})$.\ 
\end{lem}
\begin{proof}First we set $e_{I_{\mathfrak{P}}}:= (1/|I_\mathfrak{P}|)\Norm_{I_{\mathfrak{P}}}$.\ Then we have $V_\chi^{I_{\mathfrak{P}}}=V_\chi e_{I_\mathfrak{P}}$ and $V_\chi = V_\chi e_{I_{\mathfrak{P}}} \oplus V_\chi(1-e_{I_{\mathfrak{P}}})$.\ Let $M_{\phi_{\mathfrak{P}}}$ be the matrix representing the action of $1-\phi_{\mathfrak{P}}N\mathfrak{p}$ on $V_\chi^{I_{\mathfrak{P}}}$.\ Then the diagonal matrix 
\[
\begin{pmatrix}
M_{\phi_{\mathfrak{P}}} & & & \\
                             &1& &\\
                            &  & \ddots & \\
                              &  & &1
\end{pmatrix}
\]
represents the action of $1-\phi_{\mathfrak{P}} N\mathfrak{p} e_{I_{\mathfrak{P}}}$ on $V_\chi = V_\chi e_{I_{\mathfrak{P}}} \oplus V_\chi(1-e_{I_{\mathfrak{P}}})$.\ 
This implies the  equality 
\[
\det(1-\phi_{\mathfrak{P}} N \mathfrak{p} \mid V_{\chi}^{I_{\mathfrak{P}}}) = \det(1-  \phi_{\mathfrak{P}}N\mathfrak{p} e_{I_{\mathfrak{P}}}\mid V_{\chi}).\ 
\]
So we see $\epsilon_S = \nr(1-\phi_{\mathfrak{P}} N\mathfrak{p} e_{I_{\mathfrak{P}}}) \in \zeta(\mathbb{Q}[G])$.\ Since $V_{\chi}^{I_{\mathfrak{P}}}$ is a representation of the abelian group $G_{\mathfrak{P}}/I_{\mathfrak{P}}$,\ $\det(1-\phi_{\mathfrak{P}}N\mathfrak{p} \mid V_{\chi}^{I_{\mathfrak{P}}})$ is an algebraic integer.\ So $\epsilon_S$ is contained in the unique maximal order $\zeta(\Lambda^{\prime})$ in $\zeta(\mathbb{Q}[G])$.\ 
\end{proof}

\section{Statements of the non-abelian Brumer and Brumer-Stark conjectures}\label{sec:conjectures}
In this section we review the formulation of the non-abelian Brumer and Brumer-Stark conjecture by Andreas Nickel,\ for the details see \cite{Ni10}.\ 
\par
Let $K/k$ be a finite Galois CM-extension of number fields with Galois group $G$ and let $S$ and $T$ be finite sets of places of $k$ and $E_S(K)$ denote the group of $S(K)$-units of $K$.\ We set $E_S^{T}(K) := \{ x \in E_S(K) \mid x \equiv 1 \mod \prod_{\mathfrak{P} \in T(K)} \mathfrak{P}\}$.\ We refer to the following condition as $Hyp(S,T)$ ;
\begin{itemize}
 \item $S$ contains all ramifying places and all infinite places of $k$,\ 
 \item $S \cap T$ = $\emptyset$,\ 
 \item $E_S^T(K)$ is torsion free.\ 
\end{itemize}
For any fixed set $S$ which contains all ramifying places and all infinite places,\ we define 
\[
\mathfrak{A}_S := \langle \delta_T \mid \text{$Hyp(S,T)$ is satisfied} \rangle_{\zeta(\mathbb{Z}[G])}.\ 
\]
By \cite[Lemma 1.1]{Ta84},\ $\mathfrak{A}_S$ coincides with the $\mathbb{Z}[G]$-annihilator of the roots of unity in $K$ if $G$ is abelian.\ Now Nickel's formulation of the non-abelian Brumer conjecture is
\begin{conj}[$B(K/k,S)$]\label{conj:brumer}
Let $S$ be a finite set of places of $k$ which contains all ramifying places and all infinite places of $k$.\ Then 
\begin{itemize}
 \item $\mathfrak{A}_S \theta_S \subset \mathcal{ I}(\mathbb{Z}[G])$
 \item For any $x \in \mathcal{ H}(\mathbb{Z}[G])$,\ $x\mathfrak{A}_S \theta_S \subset \Ann_{\mathbb{Z}[G]}(Cl(K))$.\ 
\end{itemize}
\end{conj}
\begin{rem}
If $G$ is abelian,\ we have $\mathcal{ I}(\mathbb{Z}[G]) = \mathcal{ H}(\mathbb{Z}[G])=\mathbb{Z}[G]$ and can take $x=1$.\ So we can recover usual  Brumer's conjecture from the conjecture \ref{conj:brumer} if $G$ is abelian.\ 
\end{rem}
In this paper we actually study the following weaker version of the above conjecture.\ 
\begin{conj}[$B_w(K/k,S)$]\label{conj:wbrumer}
Let $S$ be a finite set of places of $k$ which contains all ramifying places and all infinite places of $k$ and let $\Lambda^{\prime}$ be a maximal $\mathbb{Z}$-order in $\mathbb{Q}[G]$ which contains $\mathbb{Z}[G]$.\ Then 
\begin{itemize}
 \item $\mathfrak{A}_S \theta_S \subset \zeta(\Lambda^{\prime})$
 \item For any $x \in \mathfrak{F}(\mathbb{Z}[G])$,\ $x\mathfrak{A}_S \theta_S \subset \Ann_{\mathbb{Z}[G]}(Cl(K))$.\ 
\end{itemize}
\end{conj}
\begin{rem}\label{rem:brumer}
Even if $G$ is a nontrivial abelian group,\ we always have $\Lambda^{\prime} \supsetneq \mathbb{Z}[G]$.\ Moreover,\ $\mathfrak{F}(G)$ does not contain the element $1$.\ Hence we can not recover the usual Brumer's conjecture from the conjecture \ref{conj:wbrumer} even in the case $G$ is abelian.\ Roughly speaking,\ Conjecture \ref{conj:wbrumer} says $|G|\theta^{T}_{K/k,S}$ annihilates $Cl(K)$ if $G$ is abelian.\ 
\end{rem}
Replacing $\mathbb{Z}$,\ $\mathbb{Q}$ and $Cl(K)$ with $\mathbb{Z}_p$,\ $\mathbb{Q}_p$ and $Cl(K) \otimes \mathbb{Z}_p$ respectively,\ we can decompose $B(S,K/k)$ resp. $B_w(S,K/k)$ into local conjectures $B(S,K/k,p)$ resp.\ $B_w(S,K/k,p)$.\ 
\par 
We call $\alpha \in K^{*}$ an anti-unit if $\alpha^{1 + j} = 1$ and set $w_K = \nr(|\mu(K)|)$.\ We remark that $w_k$ is no longer a rational integer but an element in $\zeta(\Lambda^{\prime})$ of the form $\sum_{\chi \in \Irr G}|\mu(K)|^{\chi(1)}e_\chi$.\ We define 
\[
 S_\alpha := \{ \mathfrak{p}\mid \text{$\mathfrak{p}$ is a prime of $k$}\  \text{and}\  \mathfrak{p}\ \text{divides}\ N_{K/k}\alpha \}
\] 
where $N_{K/k}$ is the usual  norm of $K$ over $k$.\ Then the non-abelian Brumer-Stark conjecture asserts 
\begin{conj}[$BS(K/k,S)$]\label{conj:brumer-stark}
Let $S$ be a finite set of places which contains all ramifying places and all infinite places of $k$.\ Then 
\begin{itemize}
 \item $w_K \theta_{K/k,S} \in \mathcal{ I}(\mathbb{Z}[G]) $
 \item For any fractional ideal $\mathfrak{A}$ of $K$ and for each $x \in \mathcal{ H}(\mathbb{Z}[G])$,\ there exists an anti-unit $\alpha = \alpha(\mathfrak{A},S,x)$ such that $\mathfrak{A}^{xw_K \theta_{K/k,S}} = (\alpha)$.\ 
\end{itemize}
Moreover,\ for any finite set T of places of $k$ which satisfies $Hyp(S \cup S_{\alpha},T)$,\ there exists $\alpha_T \in E_{S_\alpha}^{T}(K)$ such that 
\begin{equation}\label{eq:starkunit}
 \alpha^{z \delta_T} = \alpha_T^{zw_K}
\end{equation}
for each $z \in \mathcal{ H}(\mathbb{Z}[G])$.\ 
\end{conj}
\begin{rem}
If $G$ is abelian,\ we can take $x=z=1$.\  By \cite[Proposition 1.2]{Ta84},\ the statement (\ref{eq:starkunit}) on $\alpha$ is equivalent to the assertion that $K(\alpha^{1/w_K})/k$ is abelian.\ Hence we can regard Conjecture \ref{conj:brumer-stark} as a non-abelian generalization of the usual Brumer-Stark conjecture.\  
\end{rem}
As well as the non-abelian Brumer conjecture,\ we treat the following weaker version of the non-abelian Brumer-Stark conjecture.\ 
\begin{conj}[$BS_w(K/k,S)$]\label{conj:wbrumer-stark}
Let $S$ be a finite set of places which contains all ramifying places and all infinite places of $k$ and let $\Lambda^{\prime}$ be a maximal $\mathbb{Z}$-order in $\mathbb{Q}[G]$ which contains $\mathbb{Z}[G]$.\ Then 
\begin{itemize}
 \item $w_K \theta_{K/k,S} \in \zeta(\Lambda^{\prime})$
 \item For any fractional ideal $\mathfrak{A}$ of $K$ and for each $x \in \mathfrak{F}(\mathbb{Z}[G])$,\ there exists an anti-unit $\alpha = \alpha(\mathfrak{A},S,x)$ such that $\mathfrak{A}^{xw_K \theta_{K/k,S}} = (\alpha)$.\
\end{itemize}
Moreover,\ for any finite set T of places of $k$ which satisfies $Hyp(S \cup S_{\alpha},T)$,\ there exists $\alpha_T \in E_{S_\alpha}^{T}(K)$ such that 
\begin{equation}\label{eq:starkunit2}
 \alpha^{z \delta_T} = \alpha_T^{zw_K}
\end{equation}
for each $z \in \mathfrak{F}(\mathbb{Z}[G])$.\ 
\end{conj}
\begin{rem}
For the same reason as remark \ref{rem:brumer},\ we can not recover the usual Brumer-Stark conjecture from the conjecture \ref{conj:wbrumer-stark} in the case $G$ is abelian.\ 
\end{rem}
Replacing $\mathbb{Z}$,\ $\mathbb{Q}$ and $\mathfrak{A}$ with $\mathbb{Z}_p$,\ $\mathbb{Q}_p$ and $\mathfrak{A}$ whose class in $Cl(K)$ is of $p$-power order respectively and in the equation (\ref{eq:starkunit}),\ (\ref{eq:starkunit2}) replacing $\omega_K$ with $\omega_{K,p} := \nr(|\mu_K \otimes \mathbb{Z}_p|)$,\ we can decompose $BS(S,K/k)$ resp. $BS_w(S,K/k)$ into local conjectures $BS(S,K/k,p)$ resp.\ $BS_w(S,K/k,p)$.\ 
\par
In the abelian case,\ the Brumer-Stark conjecture implies Brumer's conjecture,\ and the same claim holds in non-abelian cases as follows:
\begin{lem}[\cite{Ni11},\ Lemma 2.9]\label{lem:relation} \ 
\begin{itemize}
 \item $BS(K/k,S)$ (resp. $BS(K/k,S,p)$) implies $B(K/k,S)$ (resp.\ $B(K/k,S,p)$)
 \item $BS_w(K/k,S)$ (resp.\ $BS_w(K/k,S,p)$) implies $B_w(K/k,S,)$ (resp.\ $B_w(K/k,S,p)$).\ 
\end{itemize}
\end{lem}
For the local conjectures,\ we can state the relation between usual conjectures and weaker conjectures as follows:
\begin{lem}\label{lem:relation2}
If $\mathcal{ I}(\mathbb{Z}_p[G]) = \zeta(\Lambda^{\prime})$ and the degrees of all irreducible characters of $G$ are prime to $p$,\ 
\begin{itemize}
 \item $B(K/k,S,p)$ holds if and only if $B_w(K/k,S,p)$ holds,\ 
 \item $BS(K/k,S,p)$ holds if and only if $BS_w(K/k,S,p)$ holds.\ 
\end{itemize}

\end{lem}
\begin{proof}
If $p$ does not divide the order of $G$ (in this case,\ the degrees of irreducible characters are automatically prime to $p$,\ since they have to divide the order of $G$), by \cite[Lemma 2.5 and Lemma 2.8]{Ni11},\ the equivalences hold.\ If $p$ divides the order of $G$,\ by Proposition \ref{prop:conductor},\ we have ${\mathcal H}(\mathbb{Z}_p[G]) = \mathfrak{F}(\mathbb{Z}_p[G])$,\ and hence we get the equivalences.\ 
\end{proof}
We let $D_{n}$ denote the dihedral group of order $n$ for any even natural number $n > 0$.\ Then as an application of Lemma \ref{lem:relation2},\ we get the following.\ 
\begin{lem}\label{lem:relation3}
Let $K/k$ be a finite Galois extension whose Galois group is isomorphic to $D_{4p}$ for an odd prime $p$.\ Then we have 
\begin{itemize}
 \item $B(K/k,S,l)$ holds if and only if  $B_w(K/k,S,l)$ holds,\
 \item $BS(K/k,S,l)$ holds if and only if $BS_w(K/k,S,l)$ holds.\ 
\end{itemize}
for any odd prime $l$.\ 
\end{lem}
\begin{proof}
It is enough to treat the case $l=p$.\ First we recall that $D_{4p}$ is isomorphic to $\mathbb{Z}/2\mathbb{Z} \times D_{2p}$.\ We set $G=\mathbb{Z}/2\mathbb{Z} \times D_{2p}$ and $j$ denotes the generator of $\mathbb{Z}/2\mathbb{Z}$.\ Since we have 
\[
\nr_{\mathbb{Q}_p[G]}(\frac{1+j}{2}) = \frac{1+j}{2} \text{\ and\ } \nr_{\mathbb{Q}_p[G]}(\frac{1-j}{2}) = \frac{1-j}{2}
\]
we also have 
\begin{equation}\label{eq:moduleI}
 {\mathcal I}(\mathbb{Z}_p[G]) = {\mathcal I}(\mathbb{Z}_p[D_{2p}])\frac{1+j}{2} \oplus   {\mathcal I}(\mathbb{Z}_p[D_{2p}])\frac{1-j}{2}.\ 
\end{equation}
By \cite[Example 6.22]{JN},\ ${\mathcal I}(D_{2p}) = \zeta(\Lambda_{D_{2p}}^{\prime})$ where $\Lambda_{D_{2p}}^{\prime}$ is a maximal $\mathbb{Z}_p$- order in $\mathbb{Q}_p[D_{2p}]$ which contains $\mathbb{Z}_p
[D_{2p}]$.\ Combining this fact with (\ref{eq:moduleI}),\ we have 
\[
 {\mathcal I}(\mathbb{Z}_p[G]) = \zeta(\Lambda_{D_{2p}}^{\prime})\frac{1+j}{2} \oplus   \zeta(\Lambda_{D_{2p}}^{\prime})\frac{1-j}{2} = \zeta(\Lambda^{\prime})
\]
where we set $\Lambda^{\prime} = \Lambda_{D_{2p}}^{\prime}\frac{1+j}{2} \oplus \Lambda_{D_{2p}}^{\prime}\frac{1-j}{2}$,\ which is a maximal order in $\mathbb{Q}_p[G]$ which contains $\mathbb{Z}_p[G]$.\ By Lemma \ref{lem:relation2},\ this completes the proof.\ 
 \end{proof}

\section{Statement and Proof of main Theorems}\label{sec:main}
In this section we prove Theorem \ref{thm:brumer} and Theorem \ref{thm:brumer-stark} which are our main theorems (the precise versions of Theorem \ref{thm:brumer0} and \ref{thm:brumer-stark0} in Introduction).\ 
\par
For each Galois extension $K/k$ whose Galois group $G$ is monomial,\ first we define the irreducible characters of $G$ and subextensions of $K/k$ as follows: Let $\chi_1,\ \chi_2,\ \dots,\ \chi_r$ be the irreducible characters of $G$ and for each $i \in \{1,\ 2,\ \dots,\ r\}$,\ we assume the character $\chi_i$ is induced by $1$-dimensional characters $\phi_{i,1},\ \phi_{i,2},\ \dots,\ \phi_{i,{s_i}}$ of a subgroup $H_i$ of $G$,\ that is,\ $\chi_i = \Ind \phi_{i,j}$ for all $j \in\{1,\ 2,\ \dots s_i\}$.\ We set $k_i := K^{H_i}$ and $K_{i,j} := K^{\ker \phi_{i,j}}$ (since $\phi_{i,j}$ is $1$-dimensional,\ $K_{i,j}/k_i$ is an abelian extension).\ We let $\phi_{i,j}^{\prime}$ be the character of $\Gal(K_{i,j}/k_i)$ whose inflation to $\Gal(K/k_i)$ is $\phi_{i,j}$.\ We set
\begin{eqnarray}
\mathbb{K}:=& \{ K_{1,1}/k_1,\ K_{1,2}/k_1,\ \dots,\ K_{1,s_1}/k_1,\ \nonumber \\
                &  \ K_{2,1}/k_2,\ K_{2,2}/k_2,\ \dots,\ K_{2,s_2}/k_2,\ \nonumber \\
                &   \cdots \nonumber \\
                &  \ K_{r,1}/k_s,\ K_{r,2}/k_s,\ \dots,\ K_{r,s_r}/k_r\} \label{eq:list}.\ 
\end{eqnarray} 
Finally,\ we fix representatives $\phi_{i} \in \{\phi_{i,1},\ \phi_{i,2},\ \dots,\ \phi_{i,s_{i}}\}$ and $K_{i} \in \{K_{i,1},\ K_{i,2},\ \dots,\ K_{i,s_{i}}\}$
\par
Now we can state and prove the following main theorems.\ 
\begin{thm}\label{thm:brumer}Let $p$ be a prime and $S$ a finite set of places of $k$ which contains all infinite places.\ We assume $G$ is monomial.\ If Brumer's conjecture (resp. the $p$-part of Brumer's conjecture) is true for all  subextensions in $\mathbb{K}$, the weak non-abelian Brumer conjecture (resp. the $p$-part of the weak non-abelian Brumer conjecture) is true for $K/k$ and $S$.\ 
\end{thm}
\begin{thm}\label{thm:brumer-stark}
The statement of Theorem \ref{thm:brumer} holds,\ with ``Brumer'' replaced by ``Brumer-Stark'' throughout.\ 
\end{thm}
\begin{rem}{\rm (1)} The following proofs show that $S$ does not have to contain the ramifying places of $k$ to deduce the weak non-abelian conjectures from abelian versions.\ \\
{\rm (2)} In fact we need weaker annihilation results than the full Brumer's conjecture or Brumer-Stark conjecture for abelian subextensions.\ To deduce the non-abelian Brumer conjecture,\ we need the annihilation results (\ref{eq:weakann2}) (see the proof of Theorem \ref{thm:brumer}),\ and to deduce the non-abelian  Brumer-Stark conjecture,\ we need the annihilation results (\ref{eq:weakann3}) and (\ref{eq:weakann4}) (see the proof of Theorem \ref{thm:brumer-stark}).\ 
\end{rem}
Before proving these theorems,\ we prepare the following lemma:\ 
\begin{lem}\label{lem:stic}Let $K/k$ be a finite Galois CM-extension of number fields with Galois group $G$.\ Let $S$ be a finite set of places of $k$ which contains all infinite places of $k$ and $T$ be another finite set of places of $k$ such that $S \cap T = \emptyset$.\ We choose a maximal $\mathbb{Z}$-order $\Lambda^{\prime}$ in $\mathbb{Q}[G]$ which contains $\mathbb{Z}[G]$.\ If $G$ is a monomial group,\ we have  
\begin{equation}\label{eq:eq2}
\theta_{S,K/k}^{T} = \sum_{i=1}^{r} \epsilon_{\chi_i,S} \delta_T(\chi_{i})\phi^{\prime}_{i}(\theta_{K_{i}/k_i}) e_{\chi_i}.\ 
\end{equation}
Furthermore if  $T$ satisfies $Hyp(S \cup S_{ram},T)$,\ $\theta_{S,K/k}^{T}$ is contained in $\zeta(\Lambda^{\prime})$.\ 
\end{lem}
\begin{rem}
The above lemma says that $S$ does not have to contain the ramifying places of $k$ for $\theta_{S,K/k}^{T}$ to lie in $\zeta(\Lambda^{\prime})$ in the monomial case.\ In \cite{Nic},\ Nickel showed a stronger result for $K/k$ whose Galois group is monomial but requires the condition $S$ to contain all the ramifying places.\ 
\end{rem}
\begin{proof}
Since $\theta_{K/k,S} = \epsilon_S\theta_{K/k,S_{\infty}}$ and $\epsilon_S \in \zeta(\Lambda^{\prime})$ by Lemma \ref{lem:epsilon},\ it is sufficient to show the equality and inclusion in Lemma \ref{lem:stic} for the case $S=S_{\infty}$.\ 
Since Artin $L$-function does not change by the induction and inflation of characters,\ we have
\begin{eqnarray}
\chi_i(\theta_{K/k,S_{\infty}}^{T}) = \chi_i(1)\delta_T(\chi_i)L_{S_\infty}(0,\overline{\chi_i},K/k) &=& \chi_i(1)\delta_T(\chi_i)L_{S_\infty}(0,\overline{\phi_{i}},K/k_i)  \nonumber \\
                                                                                                                                  &=& \chi_i(1)\delta_T(\chi_i)L_{S_\infty}(0,\overline{\phi_{i}^{\prime}},K_{i}/k_i)  \nonumber \\
                                                                                                                           &=& \chi_i(1)\delta_T(\chi_i)\phi_{i}^{\prime}(\theta_{K_{i}/k_i}).\ \label{eq:second}
\end{eqnarray}
The  equation (\ref{eq:second}) imply the equality (\ref{eq:eq2}).\ Now we assume $T$ satisfies $Hyp(S \cup S_{ram},T)$.\ 
Let $\mathfrak{p}_1^{\prime},\ \mathfrak{p}_2^{\prime},\ \dots,\ \mathfrak{p}_{l_\mathfrak{p}}^{\prime}$ be the primes of $k_i$ above $\mathfrak{p} \in T$ and $f_{\mathfrak{p}_1^{\prime}},\ f_{\mathfrak{p}_2^{\prime}},\ \dots,\ f_{\mathfrak{p}^{\prime}_{l_{\mathfrak{p}}}}$ be their residue degree.\ Then we have
\[
\delta_T(\chi_i) = \prod_{\mathfrak{p} \in T}\det(1-\phi_{\mathfrak{P}}^{-1}N\mathfrak{p}\mid V_{\chi_i}) = \prod_{\mathfrak{p} \in T} \prod_{m=1}^{l_{\mathfrak{p}}} \det(1-\phi_{\mathfrak{P}}^{-f_{\mathfrak{p}_m^{\prime}}}N\mathfrak{p}^{f_{\mathfrak{p}_m^{\prime}}}\mid V_{\phi_{i}}).\ 
\] 
We define $\delta_T^{\prime} := \prod_{\mathfrak{p} \in T} \prod_{m=1}^{l_{\mathfrak{p}}} (1-\phi_{\mathfrak{P}}^{-f_{\mathfrak{p}_m^{\prime}}}N\mathfrak{p}^{f_{\mathfrak{p}_m^{\prime}}}) $.\ Then $\delta_T^{\prime}$ is a $\mathbb{Z}[H_i]$-annihilator of $\mu(K)$ and its restriction $\delta_T^{\prime} |_{K_i}$ is a $\mathbb{Z}[\Gal(K_i/k_i)]$-annihilator of $\mu(K_i)$.\ Hence the product $\delta_T^{\prime} |_{K_i}\theta_{K_i/k_i}$ lies in $\mathbb{Z}[\Gal(K_i/k_i)]$ and 
\[
\phi_{i}^{\prime}(\delta_T^{\prime} |_{K_i}\theta_{K_{i}/k_i}) = \delta_T(\chi_i) \phi_{i}^{\prime}(\theta_{K_{i}/k_i})
\]
is an algebraic integer.\ This implies that  
\[
\theta_{K/k,S_{\infty}}^{T} \in \bigoplus_{\chi \in \Irr G/\sim} \mathfrak{o}_\chi = \zeta(\Lambda^{\prime}).\ 
\]
This completes the proof.\ 
\end{proof}

\begin{proof}[Proof of Theorem \ref{thm:brumer}]
Since $\mathfrak{F}(\mathbb{Z}[G])$ is an ideal of $\zeta(\Lambda^{\prime})$ and $\epsilon_{S \setminus S_\infty}$ lies in $\zeta(\Lambda^{\prime})$,\ it is enough to show the claim of Theorem \ref{thm:brumer} for $S_{\infty}$.\
We take $x \in \mathfrak{F}(\mathbb{Z}[G])$.\ Then $x$ is of the form $x = \sum_{\chi \in \Irr G /\sim} \sum_{\sigma \in \Gal(\mathbb{Q}(\chi)/\mathbb{Q})} x_{\chi}^{\sigma} \pr_{\chi^{\sigma}}$ with $x_\chi \in D^{-1}(\mathbb{Q}(\chi)/\mathbb{Q})$.\ Note that by definitions of $\mathfrak{F}(\mathbb{Z}[G])$,\ $x$ lies in $\zeta(\mathbb{Z}[G])$ and by the formula (\ref{eq:cond}),\  $\sum_{\sigma \in \Gal(\mathbb{Q}(\chi)/\mathbb{Q})} x_{\chi}^{\sigma} \pr_{\chi^{\sigma}}$ also lies in $\zeta(\mathbb{Z}[G])$.\ 
By Lemma \ref{lem:conductor},\ we have
\[
 \sum_{\sigma \in \Gal(\mathbb{Q}(\chi_i)/\mathbb{Q})} x_{\chi_i}^{\sigma} \pr_{\chi_i^{\sigma}} = \sum_{\substack{\phi \in \Irr H_i/\sim,\ \\ \exists \sigma \in \Gal(\mathbb{Q}(\chi_i)/\mathbb{Q}),\ \Ind \phi = \chi_i^{\sigma}}} \sum_{g \in \Gal(\mathbb{Q}(\phi)/\mathbb{Q})} x_{\chi_i}^{g} \pr_{\phi^g}\
\]
and this element lies in $\mathfrak{F}(\mathbb{Z}[H_i])$ (hence lies in $\mathbb{Z}[H_i]$).\ Moreover,\ we have $\sum_{g \in \Gal(\mathbb{Q}(\phi)/\mathbb{Q})}x_{\chi_i}^{g}\pr_{\phi^{g}} = \sum_{g \in \Gal(\mathbb{Q}(\phi_{i,j})/\mathbb{Q})}x_{\chi_{i}}^{g}\pr_{\phi_{i,j}^{g}}$ for some $j$.\ 
\par
Let $T$ be a finite set of places $k$ such that $Hyp(S_{\infty} \cup S_{ram},T)$ is satisfied.\ Then as in the proof of Lemma \ref{lem:stic},\ $\delta_T^{\prime}$ is a $\mathbb{Z}[H_i]$-annihilator of $|\mu(K)|$.\ By the assumption that Brumer's conjecture holds for subextensions in $\mathbb{K}$,\ 
\begin{eqnarray*}
(\sum_{g \in \Gal(\mathbb{Q}(\phi_{i,j})/\mathbb{Q})}x_{\chi_{i}}^{g}\pr_{\phi_{i,j}^{g}})|_{K_{i,j}} \delta^{\prime}_{T}|_{K_{i,j}} \theta_{K_{i,j}/k_i} 
&=& \sum_{g \in \Gal(\mathbb{Q}(\phi_{i,j})/\mathbb{Q})} x_{\chi_i}^{g} \phi_{i,j}^{\prime} (\delta^{\prime}_T|_{K_{i,j}} \theta_{K_{i,j}/k_i})^{g} \pr_{\phi_{i,j}^{\prime g}} \\ 
                                                                                                                                                                         &=& \sum_{g \in \Gal(\mathbb{Q}(\phi_{i,j})/\mathbb{Q})} x_{\chi_i}^{g} \delta_T(\chi_i)^{g}\phi_{i}^{\prime} (\theta_{K_{i}/k_i})^{g} \pr_{\phi_{i,j}^{\prime g}}                                                                              
\end{eqnarray*}
annihilates $Cl(K_{i,j})$.\ Combining this with the equality $\pr_{\phi_{i,j}^{g}} = \pr_{\phi_{i,j}^{\prime g}}(\Norm_{K/K_{i,j}})$,\ 
\begin{eqnarray*}
\sum_{g \in \Gal(\mathbb{Q}(\phi_{i,j})/\mathbb{Q})} x_{\chi_{i}}^{g} \delta_T(\chi_i)^{g}\phi_{i}^{\prime} (\theta_{K_{i}/k_i})^{g} \pr_{\phi_{i,j}^{g}}
\end{eqnarray*}
annihilates $Cl(K)$.\ Therefore,\ 
\begin{eqnarray}\label{eq:weakann2}
\lefteqn{\sum_{\substack{\phi \in \Irr H_i/\sim,\ \\ \exists \sigma \in \Gal(\mathbb{Q}(\chi_i)/\mathbb{Q}),\ \Ind \phi = \chi_i^{\sigma}}} \sum_{g \in \Gal(\mathbb{Q}(\phi)/\mathbb{Q})} x_{\chi_{i}}^{g} \delta_T(\chi_i)^{g}\phi_{i}^{\prime} (\theta_{K_{i}/k_i})^{g} \pr_{\phi^{g}} \nonumber} \\
&=& \sum_{\sigma \in \Gal(\mathbb{Q}(\chi_i)/\mathbb{Q})}(\sum_{\substack{\phi \in \Irr H_i/\sim_{\chi_i},\ \\ \Ind\phi = \chi_i}} \sum_{h \in \Gal(\mathbb{Q}(\phi)/\mathbb{Q}(\chi_i))} x_{\chi_{i}}^{h} \delta_T(\chi_i)^{h}\phi_{i}^{\prime} (\theta_{K_{i}/k_i})^{h} \pr_{\phi^{h}})^{\sigma} \nonumber \\
&=& \sum_{\sigma \in \Gal(\mathbb{Q}(\chi_i)/\mathbb{Q})}x_{\chi_{i}}^{\sigma} \delta_T(\chi_i)^{\sigma}\phi_{i}^{\prime} (\theta_{K_{i}/k_i})^{\sigma} (\sum_{\substack{\phi \in \Irr H_i/\sim_{\chi_i},\ \\ \Ind\phi = \chi_i}} \sum_{h \in \Gal(\mathbb{Q}(\phi)/\mathbb{Q}(\chi_i))} \pr_{\phi^{h}})^{\sigma} \nonumber \\
&=&\sum_{\sigma \in \Gal(\mathbb{Q}(\chi_i)/\mathbb{Q})}x_{\chi_{i}}^{\sigma} \delta_T(\chi_i)^{\sigma}\phi_{i}^{\prime} (\theta_{K_{i}/k_i})^{\sigma} \pr_{\chi_i^{\sigma}} \nonumber \\
&=&(\sum_{\sigma \in \Gal(\mathbb{Q}(\chi_i)/\mathbb{Q})}x_{\chi_{i}}^{\sigma} \pr_{\chi_i^{\sigma}}) \theta_{K/k}^{T}
\end{eqnarray}
annihilates $Cl(K)$.\ Finally,\ we conclude that 
\[
(\sum_{\chi \in \Irr G /\sim} \sum_{\sigma \in \Gal(\mathbb{Q}(\chi)/\mathbb{Q})} x_{\chi}^{\sigma} \pr_{\chi^{\sigma}}) \theta^{T}_{K/k,S_{\infty}} = x \theta^{T}_{K/k,S_{\infty}} 
\]
annihilates $Cl(K)$.\  
\par
To get the proof of the $p$-part conjecture we have only to replace $\mathbb{Z},\ \mathbb{Q}$ and $Cl(K)$ with $\mathbb{Z}_p,\ \mathbb{Q}_p$ and the $p$-part of $Cl(K)$ respectively.\ 
\end{proof}

\begin{proof}[Proof of Theorem \ref{thm:brumer-stark}]We use the same notations as in the proof of Theorem \ref{thm:brumer}.\ 
Take any fractional ideal $\mathfrak{A}$ of $K$.\
Then by the assumption that the Brumer-Stark conjecture holds for ebextensions in $\mathbb{K}$,\ there exists an anti-unit $\alpha_{\phi_{i,j}} \in K_{i,j}$ such that 
\begin{equation}\label{eq:bs1}
 \Norm_{K/K_{i,j}}(\mathfrak{A})^{|\mu(K_{i,j})|\theta_{K_{i,j}/k_i}} = (\alpha_{\phi_{i,j}})
\end{equation}
and 
\begin{equation}\label{eq:bs2}
 \alpha_{\phi_{i,j}}^{\delta_T^{\prime}|_{K_{i,j}}} = \alpha_{T,\phi_{i,j}}^{|\mu(K_{i,j})|}
\end{equation}
for some $\alpha_{T,\phi_{i,j}} \in E_{S_{\alpha_{\phi_{i,j}}}}^{T}(K_{i,j})$.\ Letting $c_{\phi_{i,j}} := |\mu(K)|^{\chi_i(1)}/|\mu(K_{i,j})|$ act on both side of (\ref{eq:bs1}) and (\ref{eq:bs2}),\ we have 
\[
 \Norm_{K/K_{i,j}}(\mathfrak{A})^{|\mu(K)|^{\chi_i(1)}\theta_{K_{i,j}/k_i}} = (\alpha_{\phi_{i,j}}^{c_{\phi_{i,j}}})
\]
and 
\[
 \alpha_{\phi_{i,j}}^{c_{\phi_{i,j}}\delta_T^{\prime}|_{K_{i,j}}} = \alpha_{T,\phi_{i,j}}^{|\mu(K)|^{\chi_i(1)}}.\ 
\]
Hence,\ we have 
\begin{eqnarray*}
\Norm_{K/K_{i,j}}(\mathfrak{A})^{(\sum_{g} x_{\chi_i}^{g} \pr_{\phi_{i,j}^{g}})|_{K_{i,j}} |\mu(K)|^{\chi_i(1)}\theta_{K_{i,j}/k_i}} &=& \mathfrak{A}^{(\sum_{g} x_{\chi_i}^{g} \pr_{\phi_{i,j}^{g}}) |\mu(K)|^{\chi_i(1)}\theta_{K/k_i}} \\ 
                                                                                                                                                           &=& (\alpha_{\phi_{i,j}}^{c_{\phi_{i,j}}(\sum_{g} x_{\chi_i}^{g} \pr_{\phi_{i,j}^{g}})})
\end{eqnarray*}
and 
\[
 \alpha_{\phi_{i,j}}^{c_{\phi_{i,j}}(\sum_{g} x_{\chi_i}^{g} \pr_{\phi_{i,j}^{g}}) \delta_T^{\prime}|_{K_{i,j}}} = \alpha_{T,\phi_{i,j}}^{(\sum_{g} x_{\chi_{i}}^{g} \pr_{\phi_{i,j}^{g}})|\mu(K)|^{\chi_i(1)}}.\
\]
This holds for any $j$.\ If we set 
\[
 \alpha_{\chi_i} := \prod_{\substack{\phi \in \Irr H_i/\sim,\ \\ \exists \sigma \in \Gal(\mathbb{Q}(\chi_i)/\mathbb{Q}),\ \Ind \phi = \chi_i^{\sigma}}} \prod_{g \in \mathbb{Q}(\phi)/\mathbb{Q}} \alpha_{\phi}^{c_{\phi}(\sum x_{\chi_i}^{g} \pr_{\phi^{g}})}
\]
and 
\[
 \delta_{T,\chi_i} := \prod_{\substack{\phi \in \Irr H_i/\sim,\ \\ \exists \sigma \in \Gal(\mathbb{Q}(\chi_i)/\mathbb{Q}),\ \Ind \phi = \chi_i^{\sigma}}} \prod_{g \in \mathbb{Q}(\phi)/\mathbb{Q}} \alpha_{T,\phi}^{(\sum x_{\chi_{i}}^{g} \pr_{\phi^{g}})},\ 
\]
we have
\begin{eqnarray}
\mathfrak{A}^{(\sum_{\phi} \sum_{g} x_{\chi_i}^{g} \pr_{\phi^{g}}) |\mu(K)|^{\chi_i(1)}\theta_{K/k_i}} &=& \mathfrak{A}^{(\sum_{g} x_{\chi_i^{\sigma}\pr_{\chi_i^{\sigma}}}) |\mu(K)|^{\chi_i(1)} \theta_{K/k}} \nonumber \\
                                                                                                                             &=& \mathfrak{A}^{(\sum_{g} x_{\chi_i^{\sigma}\pr_{\chi_i^{\sigma}}})(\sum |\mu(K)|^{\chi_i(1)}e_{\chi_i^{\sigma}}) \theta_{K/k}} \nonumber \\    
                                                                                                                             &=& (\alpha_{\chi_i}) \label{eq:weakann3}   
\end{eqnarray}
and 
\begin{equation}\label{eq:weakann4}
 \alpha_{\chi_i}^{z (\sum \delta_T(\chi_i)^{\sigma}e_{\chi_i^{\sigma}})} = \delta_{T,\chi_i}^{z(\sum |\mu(K)|^{\chi_i(1)}e_{\chi_i^{\sigma}})}
\end{equation}
for each $z \in \mathfrak{F}(\mathbb{Z}[G])$,\ where $\sum_{\phi} \sum_{g} x_{\chi_i}^{g} \pr_{\phi^{g}}$ means
\[
\sum_{\substack{\phi \in \Irr H_i/\sim,\ \\ \exists \sigma \in \Gal(\mathbb{Q}(\chi_i)/\mathbb{Q}),\ \Ind \phi = \chi_i^{\sigma}}} \sum_{g \in \mathbb{Q}(\phi)/\mathbb{Q}} x_{\chi_i}^{g} \pr_{\phi^{g}}.\ 
\]
This holds for any $i$.\ Finally,\ we set $\alpha := \prod_{\chi \in \Irr G/\sim} \alpha_{\chi}$ and $\alpha_T := \prod_{\chi \in \Irr G/\sim} \alpha_{T,\chi}$.\ Then $\alpha$ is an anti-unit in $K^{*}$ and $\alpha_T$ lies in $E_{S_\alpha}^{T}(K)$.\ Moreover,\ we have
\[
 \mathfrak{A}^{x\omega_K \theta_{K/k}}= (\alpha)
\]
and 
\[
 \alpha^{x\delta_T} = \alpha_T^{z \omega_K} 
\]
for each $z \in \mathfrak{F}(\mathbb{Z}[G])$.\ To get the proof of the $p$-part conjecture we have only to replace $\mathbb{Z},\ \mathbb{Q},\ \mathfrak{A}$ and $\omega_K$ with $\mathbb{Z}_p,\ \mathbb{Q}_p,\ \mathfrak{A}$ whose class in $Cl(K)$ is of $p$-power order and $\omega_{K,p}$ (in suitable places),\ respectively.\ 
\end{proof}

\section{CM-extensions with group $D_{4p}$,\ $Q_{2^{n+2}}$ or $\mathbb{Z}/2\mathbb{Z} \times A_4$}\label{sec:appl}
Let $p$ be an odd prime and $n$ be a non-zero natural number.\ We let $D_{4p}$ denote the dihedral group of order $4p$,\ $Q_{2^{n+2}}$ denote the generalized quaternion group of order $2^{n+2}$ and $A_4$ denote the alternating group on $4$ letters.\ In this section,\ as an application of Theorem \ref{thm:brumer-stark},\  we prove the $l$-parts of the weak non-abelian Brumer conjecture and the weak non-abelian Brumer-Stark conjecture for an arbitrary CM-extension of number fields $K/k$ whose Galois group is isomorphic to $D_{4p}$,\ $Q_{2^{n+2}}$ or $\mathbb{Z}/2\mathbb{Z} \times A_4$,\ where $l = 2$ in the $Q_{2^{n+2}}$ case and  $l$ is an arbitrary prime which does not split in $\mathbb{Q}(\zeta_p)$ in the $D_{4p}$ case and $\mathbb{Q}(\zeta_3)$ in the $\mathbb{Z}/2\mathbb{Z} \times A_4$ case.\ In the $D_{4p}$ case,\ we can actually verify the $p$-part of the (non-weak) non-abelian Brumer-Stark conjecture by Lemma \ref{lem:relation3}.\ In \S \ref{sec:example},\ we give an explicit example of a CM-extension with group $D_{12}$ in which $\theta_{K/k,S_{\infty} \cup S_{ram}}$ does not coincide with $\theta_{K/k,S_{\infty}}$.\ 

\subsection{CM-extensions with group $D_{4p}$}\label{sec:dih}
Let $K/k$ be a finite Galois CM-extension whose Galois group is isomorphic to $D_{4p}$.\ We use the presentation $D_{4p} = \langle x,\ y \mid x^{2p} = y^2 = 1,\ yxy^{-1} = x^{-1} \rangle$ and then $D_{4p} = \{x^{k},\ yx^{k} \mid 0\leq k \leq 2p-1\}$.\ Since the center of $D_{4p}$ is $\{1,\ x^{p} \}$,\ $x^{p}$ corresponds to the unique complex conjugation $j$.\

\subsubsection{Characters of $D_{4p}$} 
As is well known,\ all the irreducible characters of $D_{4p}$ are four $1$-dimensional characters and  $p-1$ $2$-dimensional characters.\ The $1$-dimensional characters are determined by the following table:
\begin{table}[H]
\caption{$1$-dimensional characters of $D_{4p}$}
\begin{center}
{\begin{tabular}{|c||c|c|}
\hline
&$x^{k}$& $yx^{k}$  \\
\hline
$\chi_1$ &$1$&$1$ \\
\hline
$\chi_2$ &$1$&$-1$ \\
\hline
$\chi_3$ &$(-1)^{k}$& $(-1)^{k}$\\
\hline
$\chi_4$ &$(-1)^{k}$&$(-1)^{k+1}$ \\
\hline
\end{tabular}}
\end{center}
\end{table}
\noindent Since $x^{p}$ corresponds to $j$,\ the only $1$-dimensional odd characters are $\chi_3$ and $\chi_4$.\ We easily see that $\ker \chi_3$ and $\ker \chi_4$ have index $2$ and hence we can conclude $K_3$ and $K_4$ are relative quadratic extensions of $k$.\ All the $2$-dimensional odd characters are induced by the faithful odd characters of $\langle x \rangle$.\ For $m \in (\mathbb{Z}/p\mathbb{Z})^{*}$,\ let $\phi^{m}$ be the character of $\langle x \rangle$ which sends $x^2$ and $x^{p}$ to $\zeta_p^{m}$ and $-1$ respectively.\ We set $\chi_{m+4} = \Ind^{D_{4p}}_{\langle x \rangle} \phi^m$.\ Then $k_{m+4}$ = $K^{\langle x \rangle}$ for all $m$.\ Using the Frobenius reciprocity law and the fact that $\chi_m(1) = 2$ and $\chi_m(j)=-2$,\ we see that $\Res^{D_{4p}}_{\langle x \rangle} \chi_{m+4} = \phi^{m} + \phi^{-m}$ and $\Ind^{D_{4p}}_{\langle x \rangle} \phi^m=\Ind^{D_{4p}}_{\langle x \rangle} \phi^{-m}$.\ Therefore,\ the number of $2$-dimensional odd characters is $(p-1)/2$.\ Since $\phi^{m}$ and $\phi^{-m}$ are faithful,\ we see that $K_{m+4,1} = K_{m+4,2} = K$.\ 
\subsubsection{Proof of conjectures for extensions with group $D_{4p}$}
In this subsection,\ we prove the following theorem by using Theorem \ref{thm:brumer-stark}.
\begin{thm}\label{thm:brumer1}
Let $K/k$ be a finite Galois CM-extension whose Galois group is isomorphic to $D_{4p}$ and $S$ be a finite set of places of $k$ which contains all infinite places.\  Then 
\begin{description}
\item[(1)] the $p$-part of the non-abelian Brumer conjecture and the non-abelian Brumer-Stark conjecture are true for $K/k$ and $S$,\ 
\item[(2)] for each prime $l$ (including $2$) which does not split in $\mathbb{Q}(\zeta_p)$,\  the $l$-part of the weak non-abelian Brumer conjecture and the weak non-abelian Brumer-Stark conjecture are true for $K/k$ and $S$.\ 
\end{description}
\end{thm}
\noindent
\begin{rem}
{\rm (1)} In the case $k = \mathbb{Q}$,\ the above two results except the $2$-part are contained in Nickel's work \cite{Ni11},\ \cite{Nib} if we assume $\mu=0$.\ \\
{\rm (2)} If no prime above $p$ splits in $K/K^{+}$ whenever $K^{cl} \subset (K^{cl})^{+}(\zeta_p)$,\ the odd $p$-part of the above results holds unconditionally by \cite[Corollary 4.2]{Nia}.\ 
\end{rem}

The observation we made in the previous subsection tells us that we have only to verify the Brumer-Stark conjecture for two relative quadratic extensions $K_3/k$,\ $K_4/k$ and the cyclic extension $K/k_5$.\ By \cite{Ta}[\S3,\ case(c)],\ the Brumer-Stark conjecture is true for any relative quadratic extensions and hence true for $K_3/k$,\ $K_4/k$.\ In order to complete the proof of Theorem \ref{thm:brumer1},\  we have to verify the $l$-part of the Brumer-Stark conjecture for $K/k_5$ for each prime $l$ which does not split in $\mathbb{Q}(\zeta_p)$.\ However,\ the proof of Theorem \ref{thm:brumer-stark} (and Lemma \ref{lem:relation3}) tells us that we only have to verify the slightly weaker annihilation result,\ that is,\ we only need  (\ref{eq:weakann3}) and  (\ref{eq:weakann4}) for $K/k_5$.\ To do that,\ it is enough to prove the following proposition: 
\begin{prop}\label{cl:6p}
Let $l$ be a prime which does not split in $\mathbb{Q}(\zeta_p)$.\ Let $K/F$ be any cyclic CM-extension of number fields of degree $2p$.\ We assume $F$ contains $k$ so that ($F/k$ is quadratic and ) $K/k$ is CM with Galois group $D_{4p}$.\ We let $\sigma$ be a generator of the Galois group of $K^{\langle j \rangle}/F$ (hence $\Gal(K/F)=\langle \sigma j \rangle$).\  Take any element of the form $\sum_{m=1}^{(p-1)/2} x_{\chi_{m+4}} \pr_{\chi_{m+4}}$ in $\mathfrak{F}(D_{4p})$.\ Then for any fractional ideal $\mathfrak{A}$ of $K$ whose class in $Cl(K)$ is of $l$-power order,\ 
\begin{description}
 \item[(1)] $\mathfrak{A}^{\omega_K \theta_{K/F}(\sum_{m=1}^{(p-1)/2} x_{\chi_{m+4}} \pr_{\chi_{m+4}})} = (\alpha)$ for some anti-unit $\alpha \in K^{*}$,\ 
 \item[(2)] $K(\alpha^{1/\omega_{K,l}})/F$ is abelian.\
\end{description} 
where $\omega_{K,l}$ is the $l$-part of $\omega_K$ (if $l$ is odd,\ we can ignore the  claim $\alpha$ is anti-unit as in the remark just before \cite[Proposition 1.1]{GRT04}).\ 
\end{prop}
\begin{rem}
The method of the proof of this proposition is essentially the same as that of \cite[Proposition 2.2 and Proposition 2.1]{GRT04} but we do not need the classifications in loc.\ cit\  because we only need a weaker annihilation results than the full Brumer-Stark conjecture.\ 
\end{rem}
\begin{proof}[\bfseries Proof of Proposition \ref{cl:6p}]
(i) First,\ we suppose $l=2$.\ In this case,\ by \cite[Theorem3.2]{GRT04},\ Proposition \ref{cl:6p} holds for $p=3$ and exactly the same proof works for any odd prime $p$ if $2$ does not split in $\mathbb{Q}(\zeta_p)$ .\ Hence Proposition \ref{cl:6p} holds in this case.\ \\
(ii) In what follows we assume $l$ is odd.\ Let $\psi$ be the irreducible character of $\Gal(K/F)$ which sends $\sigma$ and $j$ to $1$ and $-1$ respectively.\ Then this character is the inflation of the nontrivial character $\psi^{\prime}$ of $\Gal(E/F)$ where $E$ = $K^{H}$ and $H = \langle \sigma \rangle$.\ We put 
\begin{eqnarray*}
A_K &:=& \frac{1-j}{2}(Cl(K) \otimes \mathbb{Z}_l),\  \\
A_E &:=& \frac{1-j}{2}(Cl(E) \otimes \mathbb{Z}_l).\  
\end{eqnarray*}
Then by analytic class number formula,\ we get
\begin{eqnarray}
|A_K| &=& \omega_{K,l} L(0,\psi,K/F) \prod_{j=1}^{p-1} L(0,\phi^{j},K/F)  \nonumber \\
           &=& \omega_{K,l} L(0,\psi^{\prime},E/F) N_{\mathbb{Q}(\zeta_p)/\mathbb{Q}}(L(0,\phi,K/F)),\ \label{eq:cnf1}\\
|A_E| &=& \omega_{F,l}  L(0,\psi^{\prime},E/F) \label{eq:cnf2}
\end{eqnarray}
where the equalities are considered as equalities of the $l$-part.\ If $l \not = p$,\ $|A_E|$=$|A_K^{H}|$ since $A_E$ is canonically isomorphic to $A_K^{H}$.\ If $l=p$,\ we have $|A_E|$ $\leq$ $|A_K^{H}|$ by \cite[Lemma 2.5]{GRT04}.\   Since $(\sum_{m=1}^{(p-1)/2} x_{m+4}\pr_{\chi_{m+4}} )A_K^{H}=0$,\ there is a natural surjection $A_K/A_K^{H} \twoheadrightarrow (\sum_{m=1}^{(p-1)/2} x_{m+4} \pr_{\chi_{m+4}})A_K$.\ Hence we have 
\begin{equation}\label{eq:classeq}
|(\sum_{m=1}^{\frac{p-1}{2}} x_{m+4}\pr_{\chi_{m+4}} )A_K| \leq |A_K|/|A_K^{H}| \leq |A_K|/|A_E| = \frac{\omega_{K,l}}{\omega_{E,l}}  N_{\mathbb{Q}(\zeta_p)/\mathbb{Q}}(L(0,\phi,K/F)).\ 
\end{equation}
\noindent
Since the minus part of $\mathbb{Q}_l[G]$ is isomorphic to $\mathbb{Q}_l[H]$ by sending $j$ to $-1$,\ in what follows,\  we identify the minus part of $\mathbb{Q}_l[G]$ with $\mathbb{Q}_l[H]$ just like \cite[\S 2]{GRT04} (for example $\theta_{K/F}$ will be regarded as an element of $\mathbb{Q}_l[H]$ not of $\mathbb{Q}_l[G]$).\ \\ \\
Case I. $l \not = p$.\ \\
In this case,\  the equality holds in (\ref{eq:classeq}).\ Moreover,\ we have $\omega_{K,l}/\omega_{E,l} =1$ and hence the elements $L(0,\phi^{m},K/F)$ are contained in $\mathbb{Z}_l[\zeta_p]$.\ Since $l \not = p$ we get an isomorphism 
\[
\mathbb{Z}_l[H] \cong \bigoplus_{\eta \in \Irr H/\sim} \mathbb{Z}_l[\eta]
\]
where $\eta$ runs over all irreducible characters of $H$ modulo $\Gal(\mathbb{Q}_l(\zeta_p)/\mathbb{Q}_l)$-action.\
 Hence,\ we have 
\begin{equation*}
  A_K/A_K^{H} \cong (\sum_{m=1}^{p-1} e_{\phi^{m}}) A_K \cong \bigoplus_{\eta \in \Irr H \setminus \{1\}/\sim} \mathbb{Z}_l[\eta] \otimes_{\mathbb{Z}[H]}A_K 
\end{equation*}
By assumption that $l$ does not split in $\mathbb{Q}(\zeta_p)$,\ we actually have
\begin{equation}\label{cong:hmod}
  A_K/A_K^{H} \cong \mathbb{Z}_l[\eta] \otimes_{\mathbb{Z}[H]} A_K 
\end{equation}
By (\ref{eq:classeq}),\  we have
\begin{eqnarray}
|(\sum_{m=1}^{\frac{p-2}{2}}x_{m+4}\pr_{\chi_{m+4}})A_K| \leq | A_K/A_K^{H}| = |\mathbb{Z}_l[\eta] \otimes_{\mathbb{Z}[H]}A_K| & = & [\mathbb{Z}_l[\zeta_p]:(L(0,\phi,K/F))] \nonumber \\
                                                                       &=& [\mathbb{Z}_l[\zeta_p]:(\overline{\theta}_{K/F})] \label{eq:cn3}
\end{eqnarray}
where $\overline{\theta}_{K/F}$ is the image of $\theta_{K/F}$ under the surjection $\mathbb{Z}_l[H] \twoheadrightarrow \mathbb{Z}_l[\zeta_p]$.\ 
Since we have 
\[
(1+ \sigma + \sigma^2+ \cdots \sigma^{p-1})(\sum_{m=1}^{\frac{p-2}{2}}x_{m+4}\pr_{\chi_{m+4}}) = 0,\ 
\]
we can regard $(\sum_{m=1}^{(p-1)/2}x_{m+4}\pr_{\chi_{m+4}})A_K$ as a $\mathbb{Z}_l[\zeta_p]$-module through the natural surjection $\mathbb{Z}_l[H] \twoheadrightarrow \mathbb{Z}_l[\eta] = \mathbb{Z}_l[\zeta_p]$.\  
Moreover,\ since $\sum_{m=1}^{(p-1)/2}x_{m+4}\pr_{\chi_{m+4}}A_K$ is a torsion module,\ there exists $n_{1},\ n_{2},\ \dots,\ n_{k} \in \mathbb{N}$ such that
\[
 (\sum_{m=1}^{\frac{p-2}{2}}x_{m+4}\pr_{\chi_{m+4}})A_K \cong \bigoplus_{i=1}^{k} \mathbb{Z}_l[\zeta_p] / (l)^{n_{i}}
\]
Combining the above isomorphism with (\ref{eq:cn3}),\ we have 
\[
|(\sum_{m=1}^{\frac{p-2}{2}}x_{m+4}\pr_{\chi_{m+4}})A_K| = |\bigoplus_{i=1}^{k}\mathbb{Z}_l[\zeta_p] / (l)^{n_{i}}| \leq |\mathbb{Z}_l[\zeta_p]/(\overline{\theta}_{K/F})|
\]
This inequality implies that $\theta_{K/F}$ annihilates $(\sum_{m=1}^{(p-1)/2} x_{m+4}\pr_{\chi_{m+4}})A_K$.\ Therefore,\ for any fractional ideal $\mathfrak{A}$ of $K$ whose class in $Cl(L)$ is of $l$-power order,\ $\mathfrak{A}^{\omega_{K} \theta_{K/F}(\sum_{m=1}^{(p-1)/2} x_{m+4}\pr_{\chi_{m+4}})} = (\alpha^{\omega_{K}})$ for some $\alpha \in K^{*}$ and clearly $K((\alpha^{\omega_{K}})^{1/\omega_{K,l}})/F$ is abelian.\ This completes the proof of Claim \ref{cl:6p} in this case.\ \\ \\

\noindent
Case II.\ $l=p$ and $\omega_{K,p}/\omega_{E,p} =1$.\ \\
In this case,\ by (\ref{eq:classeq}),\ we have 
\begin{equation}\label{eq:cn4}
 |(\sum_{m=1}^{\frac{p-1}{2}} x_{m+4}\pr_{\chi_{m+4}} )A_K| \leq [\mathbb{Z}_p[\zeta_p]:(\overline{\theta}_{K/F})].\ 
\end{equation}
Since $(\sum_{m=1}^{(p-1)/2}x_{m+4} \pr_{\chi_{m+4}} )A_K$ is a torsion module,\ there exists $n_1,\ n_2,\ \dots,\ n_{m} \in \mathbb{N}$ such that 
\[
 (\sum_{m=1}^{\frac{p-1}{2}} x_{m+4}\pr_{\chi_{m+4}} )A_K \cong \bigoplus_{i=1}^{m}\mathbb{Z}_p[\zeta_p] / (1-\zeta_p)^{n_i}.\ 
\]
Combining this with (\ref{eq:cn4}) we have 
\[
| (\sum_{m=1}^{\frac{p-1}{2}}x_{m+4} \pr_{\chi_{m+4}} )A_K|=|\bigoplus_{i=1}^{m}\mathbb{Z}_p[\zeta_p] / (1-\zeta_p)^{n_i}| \leq |\mathbb{Z}_p[\zeta_p]/(\overline{\theta}_{K/F})|
\]
This implies $\theta_{K/F}$ annihilates $(\sum_{m=1}^{(p-1)/2}x_{m+4} \pr_{\chi^{m+4}})A_K $.\ By the same argument as the final part of Case II,\ we obtain the conclusion in this case.\ \\

\noindent
Case III.\ $l=p$ and $\omega_{K,p}/\omega_{E,p} \not =1$\\
In this case,\ we see that $\omega_K = p^{e}$,\ $\omega_{E} = p^{e-1}$ for some $e \in \mathbb{N}$.\ Then we have 
\[
|(\sum_{m=1}^{\frac{p-1}{2}} x_{m+4}\pr_{\chi_{m+4}} )A_K| \leq [\mathbb{Z}_p[\zeta_p]:(\zeta_p-1)(\overline{\theta}_{K/F})].\
\] 
This implies that $(\sigma - 1)(\sum_{m=1}^{(p-1)/2} x_{m+4}\pr_{\chi_{m+4}}) \theta_{K/F}$ annihilates $A_K$.\ Hence for any fractional ideal $\mathfrak{A}$ of $K$ whose class in $Cl(K)$ is of $p$-power order,\ there exists some $\beta \in K$ such that 
\[
\mathfrak{A}^{\omega_{K,p}(\sigma-1)\theta_{K/F}(\sum_{m=1}^{(p-1)/2} x_{m+4}\pr_{\chi_{m+4}})}=(\beta).\ 
\]
In the last paragraph of \cite[Proposition 2.2]{GRT04},\ the authors show that if $(\sum_{j=0}^{p-1}\sigma^{j})\theta_{K/F} = 0$,\ there exists $\alpha \in \mathbb{Z}_p[H]$ such that
 \[
p^{e}\theta_{K/F} = (\sigma-1)\alpha \gamma\theta_{K/F}
\]
where $\gamma = \sigma^{p-1}+g\sigma^{p-2} +\cdots + g^{p-1}$ and $g$ is the minimal positive integer which represents the action of $\sigma$ on the $p^{e}$th-power root of unity in $K$.\ Since,\ $(\sum_{j=0}^{p-1}\sigma^{j})(\sum_{m=1}^{(p-1)/2} x_{m+4}\pr_{\chi_{m+4}}) \theta_{K/F} =0$,\ replacing $\theta_{K/F}$ by $(\sum_{m=1}^{(p-1)/2} x_{m+4}\pr_{\chi_{m+4}}) \theta_{K/F}$,\ we get
\[
p^{e}(\sum_{m=1}^{(p-1)/2}x_{m+4} \pr_{\chi_{m+4}}) \theta_{K/F} = (\sigma-1)\alpha \gamma(\sum_{m=1}^{(p-1)/2} x_{m+4}\pr_{\chi_{m+4}}) \theta_{K/F},\ 
\]
for some $\alpha \in \mathbb{Z}_p[H]$.\ This implies 
\[
\mathfrak{A}^{p^{e}\theta_{K/F}(\sum_{m=1}^{(p-1)/2} x_{m+4}\pr_{\chi_{m+4}} )} = (\beta^{\alpha \gamma}).\
\]
To conclude the proof of Case IV,\ we use the following proposition:

\begin{prop}[Proposition 1.2 d),\ \cite{Ta84}]
Let $L/k$ be an arbitrary abelian extension of number fields with Galois group G,\  $\{ \sigma_i \}_{i \in I}$ be a system of generators of $G$,\ $\zeta$ be a primitive $\omega_L$th - root of unity.\ We suppose $\sigma_i$ acts on $\zeta$ as $\zeta^{\sigma_i} = \zeta^{n_i}$.\ We take an element $\beta \in F$.\ Then for any natural number $m$,\  the following statement is equivalent to the condition that $F(\beta^{1/m})/K$ is abelian: \\
There exists a system $\{ \beta_i\}_{i \in I} \subset E_F$ such that 
\begin{equation*}
\begin{array}{lllr}
 \alpha_i^{\sigma_j - n_j} &=& \alpha_j^{\sigma_i - n_i} & \text{ for any $i,\ j \in I$},\ \\
 \beta^{\sigma_i - n_i} &=& \alpha_i^{m}  & \text{for any $i \in I$}.\ 
\end{array} 
\end{equation*}
\end{prop}
\noindent
Applying this proposition to our setting,\ we have
\[
\text{$K((\beta^{\alpha \gamma})^{1/p^{e}})/F$ is abelian }\text{if and only if } \text{there exists $\alpha \in E_K$ such that $(\beta^{\alpha \gamma})^{\sigma - g} = \alpha^{p^{e}}$}.\ 
\]
Since  $(\beta^{\alpha \gamma})^{\sigma-g}=(\beta^{\alpha})^{1-g^p}$ and $1-g^p$ is divisible by $p^{e}$,\ we can conclude $K((\beta^{\alpha \gamma})^{1/p^{e}})/F$ is abelian .\ 
\end{proof}
\noindent
\subsubsection{An example}\label{sec:example}
In this section,\ we give an explicit example of the CM-extension $K/\mathbb{Q}$ with group $D_{12}$ in which $\theta_{K/\mathbb{Q},S_{\infty} \cup S_{ram}}$ does not coincide with $\theta_{K/\mathbb{Q},S_{\infty}}$.\ 
\par
Let $\alpha$ be an element in $\overline{\mathbb{Q}}$ which is a root of the cubic equation $x^3 -9x +3 = 0$ and let $K = \mathbb{Q}(\sqrt{-2},\ \sqrt{33},\ \alpha)$.\ Then $K/\mathbb{Q}$ is a Galois extension whose Galois group $G$ is isomorphic to $\Gal(\mathbb{Q}(\sqrt{-2})/\mathbb{Q}) \times \Gal(\mathbb{Q}(\sqrt{33},\ \alpha)/\mathbb{Q}) \cong \mathbb{Z}/2\mathbb{Z} \times \mathfrak{S}_3 \cong D_{12}$ where $\mathfrak{S}_3$ is the symmetric group of degree $3$.\ We use the presentation $\mathfrak{S}_3 = \langle \sigma,\ \tau \mid \sigma^{3} = \tau^{2} = 1,\ \tau \sigma \tau = \sigma^{-1} \rangle$.\ Since the center of $G$ is $\Gal(\mathbb{Q}(\sqrt{-2})/\mathbb{Q}) \cong \mathbb{Z}/2\mathbb{Z}$,\ the generator of $\mathbb{Z}/2\mathbb{Z}$ corresponds to the unique complex conjugation $j$.\ The irreducible characters of $G$ are determined by the following character table,\ where $\{ \cdot \}$ indicates conjugacy classes:
\begin{table}[H]
\caption{The character table of $G$}
\begin{center}
{\begin{tabular}{|c||c|c|c|c|c|c|}
\hline
&$\{1\}$ & $\{ \sigma \}$ & $\{ \tau \}$ & $\{ j \}$ & $\{ \sigma j \}$ & $\{ \tau j \}$ \\
\hline
$\chi_1$ & $1$ & $1$ & $1$ & $1$ & $1$ & $1$ \\
\hline
$\chi_2$ & $1$ & $1$ & $1$ & $-1$ & $-1$ & $-1$ \\
\hline
$\chi_3$ & $1$ & $1$ & $-1$ & $1$ & $1$ & $-1$ \\
\hline
$\chi_4$ & $1$ & $1$ & $-1$ & $-1$ & $-1$ & $1$ \\
\hline
$\chi_5$ & $2$ & $-1$ & $0$ & $2$ & $-1$ & $0$ \\
\hline
$\chi_6$ & $2$ & $-1$ & $0$ & $-2$ & $1$ & $0$ \\
\hline
\end{tabular}}
\end{center}
\end{table}
\noindent
From the above table,\ we see that the only odd characters are $\chi_2,\ \chi_4$ and $\chi_6$.\ Since $\ker \chi_2 = \Gal(K/\mathbb{Q}(\sqrt{-2}))$,\ $\ker \chi_4 = \Gal(K/\mathbb{Q}(\sqrt{-66}))$,\ we see that $K_2 = \mathbb{Q}(\sqrt{-2}),\ K_4=\mathbb{Q}(\sqrt{-66})$.\ Let $\phi_6$ be an irreducible character of $\Gal(K/\mathbb{Q}(\sqrt{33})) = \langle \sigma j \rangle$ which sends $\sigma$ and $j$ to $\zeta_3$ and $-1$ respectively where $\zeta_3$ is a primitive $3$rd root of unity in $\overline{\mathbb{Q}}$.\ Then $\chi_6$ is the induced character of $\phi_6$,\ that is,\ $\chi_6 = \Ind^{G}_{\Gal(K/\mathbb{Q}(\sqrt{33}))} (\phi_6)$ and hence $k_6 = \mathbb{Q}(\sqrt{33})$ and $K_6 = K$.\ Then by using Pari/GP,\ we have 
\[
\begin{array}{ccccc}
L_{S_{\infty}}(0,K/k,\chi_2) &=& L_{S_{\infty}}(0,\chi_2^{\prime},K_2/\mathbb{Q}) &=&1,\  \\
L_{S_{\infty}}(0,K/k,\chi_4) &=& L_{S_{\infty}}(0,\chi_4^{\prime},K_4/\mathbb{Q}) &=&8,\  \\
L_{S_{\infty}}(0,K/k,\chi_6) &=& L_{S_{\infty}}(0,\phi_2, K/k_6) &=& 48.\ 
\end{array}
\]
Therefore we have 
\begin{eqnarray*}
 \theta_{K/\mathbb{Q},S_{\infty}} &=& 1\cdot e_{\chi_2} + 8 \cdot e_{\chi_4} + 48 \cdot e_{\chi_6} \\
                                               &=& \frac{1}{4}(1-j)(67 -29(\sigma + \sigma^2) -7(\tau + \sigma \tau + \sigma^2 \tau)).\ 
\end{eqnarray*}
As in the proof of  Lemma \ref{lem:relation3},\ ${\mathcal I}(\mathbb{Z}_l[G]) = \zeta(\Lambda^{\prime})$ for any odd $l$ where $\Lambda^{\prime}$ is a maximal $\mathbb{Z}_l$-order in $\mathbb{Q}_l[G]$ which contains $\mathbb{Z}_l[G]$.\ Therefore,\ this element lies in ${\mathcal I}(\mathbb{Z}_l[G])$ for any odd $l$ (especially including $3$).\ However,\ by \cite[Proposition 4.3 and 4.8]{JN},\ $\theta_{K/\mathbb{Q},S_{\infty}}$ does not lie in ${\mathcal I}(\mathbb{Z}_{2}[G])$.\ 
\par 
Next we compute $\theta_{K/\mathbb{Q}} = \theta_{K/\mathbb{Q}, S_{\infty} \cup S_{ram}}$.\ All the primes which ramify in $K/\mathbb{Q}$ are $2,\ 3$ and $11$.\ Taking suitable primes $\mathfrak{P}_2,\ \mathfrak{P}_3$ and $\mathfrak{P}_{11}$ of $K$ above $2,\ 3$ and $11$ respectively,\ their decomposition groups and inertia groups are  determined as follows : 
\[
\begin{array}{ccccc}
 G_{\mathfrak{P}_2} &=& \Gal(K/\mathbb{Q}(\sqrt{33}))            & \cong & \langle \sigma j \rangle,\  \\
 I_{\mathfrak{P}_2} &=& \Gal(K/\mathbb{Q}(\sqrt{33},\alpha)) & \cong & \langle j \rangle,\  \\
 G_{\mathfrak{P}_3} &=& \Gal(K/\mathbb{Q}(\sqrt{-2}))            & \cong & \mathfrak{S}_3,\  \\
 I_{\mathfrak{P}_3} &=& \Gal(K/\mathbb{Q}(\sqrt{-2}))           & \cong & \mathfrak{S}_3,\  \\
 G_{\mathfrak{P}_{11}} &=& \Gal(K/\mathbb{Q}(\sqrt{-2},\ \alpha))            & \cong & \langle \tau \rangle,\  \\
 I_{\mathfrak{P}_{11}} &=& \Gal(K/\mathbb{Q}(\sqrt{-2},\ \alpha))           & \cong & \langle \tau \rangle.\  \\
\end{array}
\]
This fact implies that $\epsilon_{\psi_2, S_{ram}} = \epsilon_{\chi_2,S_{ram}} =0$ and $\epsilon_{\psi_4,S_{ram}} =1$.\ Hence,\ we have 
\begin{eqnarray*}
 \theta_{K/\mathbb{Q}} &=& 8 \cdot e_{\chi_4}  \\
                                 &=& \frac{2}{3}(1-j)(1 + \sigma + \sigma^2 - \tau - \sigma \tau - \sigma^2 \tau).\ 
\end{eqnarray*}
This element is a zero divisor in $\mathbb{Q}[G]^{-} = \frac{(1-j)}{2} \mathbb{Q}[G]$.\ Since $\theta_{K/\mathbb{Q},S_{\infty}}$ is not a zero-divisor in $\mathbb{Q}[G]^{-}$,  $\theta_{K/\mathbb{Q}}$ is essentially different from $\theta_{K/\mathbb{Q}.S_{\infty}}$.\ 
\par
Moreover,\ this element lies in ${\mathcal I}(\mathbb{Z}_2[G])$ and hence lies in ${\mathcal I}(\mathbb{Z}[G])$.\ This fact tells us that for general Galois extensions $K/k$ whose Galois group $G$ ,\ $S$ has to contain ramifying places for Stickelberger elements to lie in ${\mathcal I}(\mathbb{Z}[G])$(at least to lie in ${\mathcal I}(\mathbb{Z}_2[G])$).\ 
\subsection{CM-extensions with group $Q_{2^{n+2}}$}\label{sec:quat}
Let $K/k$ be a finite Galois extension whose Galois group is isomorphic to the quaternion group $Q_{2^{n+2}}$ of order $2^{n+2}$.\ We use the presentation $Q_{2^{n+2}} = \langle x,\ y \mid x^{2^{n}} =y^{2},\ \ x^{2^{n+1}} = 1,\ yxy^{-1} = x^{-1} \rangle$.\ Since the center of $Q_{2^{n+2}}$ is $\{1,\ x^{2^{n}}\}$,\ $x^{2^{n}}$ corresponds to the unique complex conjugation $j$.\ 
\subsubsection{Characters of $Q_{2^{n+2}}$}
$Q_{2^{n+2}}$ has two types of irreducible characters.\ One type is given through the natural surjection $Q_{2^{n+2}} \twoheadrightarrow Q_{2^{n+2}}/\langle x^n \rangle \simeq D_{2^{n+1}}$.\ Clearly,\ characters which are given in this way are even characters.\ The other type is two dimensional characters which are induced by the faithful odd characters of $\langle x \rangle$ (in fact,\ a character of $\langle x \rangle$ is faithful if and only if it is odd).\ Let $\phi$ be the character of $\langle x \rangle $ which sends $x$ and $x^n$ to $\zeta_{2^{n+1}}$ and $-1$ respectively.\ Then all faithful odd characters are of the form $\phi^m$ for $m \in (\mathbb{Z}/2^{n+1}\mathbb{Z})^{*}$.\ We set $\chi_m := \Ind^{Q_{2^{n+2}}}_{\langle x \rangle}\phi^m$.\ Then we have $\chi_m = \chi_{-m}$ and $k_m=k^{\langle x \rangle}$ for all $m$.\ Since $\phi^m$ is faithful,\  we conclude $K_{m,1} = K_{m,2} = K$.\ 

\subsubsection{Proof of conjectures for extensions with group $Q_{2^{n+2}}$}
First,\ we define $M := \{ a \mid 1\leq a \leq 2^{n+1},\ \text{a is odd}\},\ M^{+} :=\{ a \mid 1\leq a \leq 2^{n-1},\ \text{a is odd}\}$.\ In this subsection,\ we prove the following theorem by using Theorem \ref{thm:brumer-stark}.\ 
\begin{thm}\label{thm:brumer2}
Let $K/k$ be a finite Galois CM-extension whose Galois group is isomorphic to $Q_{2^{n+2}}$ and $S$ be a finite set of places of $k$ which contains all infinite places.\ Then the $2$-part of the weak non-abelian Brumer conjecture and the weak non-abelian Brumer-Stark conjecture are true for $K/k$ and $S$,\ 
\end{thm}
\begin{rem}
{\rm (1)} If no prime above $p$ splits in $K/K^{+}$ whenever $K^{cl} \subset (K^{cl})^{+}(\zeta_p)$,\ the odd $p$-part of the above result holds by \cite[Corollary  4.2]{Nia}.\ \\
{\rm (2)} Since all the subgroups of $Q_{2^{n+2}}$ are normal and all the odd representations are faithful,\ $\theta_{K/k,S_{\infty} \cup S_{ram}}$ always coincide with $\theta_{K/k,S_{\infty}}$.\ 
\end{rem}
The observation in the previous section tells us that we have to verify the $l$-part of the Brumer-Stark conjecture for $K/K^{\langle x \rangle}$ for $l$ which does not split in $\mathbb{Q}(\zeta_{2^{n+1}})$.\ As in the previous section,\ however,\ we only need slightly weaker annihilation results (\ref{eq:weakann3}) and (\ref{eq:weakann4}).\ To verify those weaker results,\ it is enough to prove the following:

\begin{prop}\label{thm:brumerofcyclic}
Let $K/F$ be a cyclic CM-extension of degree $2^{n+1}$.\ We assume $F$ contains $k$ so that ($F/k$ is quadratic and)  $K/k$ is CM with Galois group $Q_{2^{n+2}}$.\  Take any element of the form $\sum_{m \in M^{+}} x_{\chi_m} \pr_{\chi_m}$ in $\mathfrak{F}(Q_{2^{n+2}})$.\ Then for any fractional ideal $\mathfrak{A}$ of $K$ whose class in $Cl(K)$ is of $2$-power order,\ 
\begin{description}
 \item[(1)] $\mathfrak{A}^{\omega_K \theta_{K/F}(\sum_{m \in M^{+}} x_{\chi_m} \pr_{\chi_m})} = (\alpha)$ for some anti-unit $\alpha \in K^{*}$,\ 
 \item[(2)] $K(\alpha^{1/\omega_{K,2}})/F$ is abelian.\
\end{description} 
where $\omega_{K,2}$ is the $2$-part of $\omega_K$.\ 
\end{prop}
\noindent
Before proving the above theorem,\ we prove the following lemma:  
\begin{lem}\label{lem:roots}
Let $K/F$ be a cyclic CM-extension of degree $2^{n+1}$ which is contained in some $Q_{2^{n+2}}$-extension.\ Then all roots of unity in $K$ are $\pm 1$.\ 
\end{lem}
\begin{proof}[\bfseries Proof of Lemma \ref{lem:roots}]
Let $\zeta$ be a primitive $\omega_K$th roots of unity in $K$ and assume $x(\zeta) = \zeta^{c_x}$ and $y(\zeta) = \zeta^{c_y}$ for some $c_x,\ c_y \in (\mathbb{Z}/\omega_K \mathbb{Z})^{*}$.\ Then we have $yxy^{-1}(\zeta) = \zeta^{c_y^{-1}c_xc_y} = \zeta^{c_x}$.\ On the other hand $yxy^{-1} = x^{-1}$,\ so we have $yxy^{-1}(\zeta) = \zeta^{c_x^{-1}}$.\ Hence we see that 
\[
c_x \equiv c_x^{-1} \mod \omega_K \Leftrightarrow c_x^2 \equiv 1 \mod \omega_K.\ 
\]
Therefore,\ we have $x^2(\zeta) = \zeta$ and hence $x^{2n}(\zeta) = \zeta$.\ This implies $\zeta$ lies in $K^{+}$.\ 
\end{proof}
\begin{proof}[\bfseries Proof of Proposition \ref{thm:brumerofcyclic}]
We define the group $I_K^{+}$ of the ambiguous ideals by 
\[
I_K^{+} := \{ \mathfrak{A} \mid \text{$\mathfrak{A}$ is an ideal of $K$ such that $\mathfrak{A}^{j} = \mathfrak{A}$} \}
\]
where $j$ is the unique complex conjugation in $\Gal(K/F)$.\ Also we define $A_K := \Coker(I_K^{+} \rightarrow Cl(K)) \otimes \mathbb{Z}_2$.\ Then by Sands's  formula \cite[Proposition 3.2]{Sa} (also see \cite[\S 3]{GRT04}),\ we have 
\begin{equation}\label{eq:sands}
\omega_K \theta_{K^{+}/K} = 2^{[K^{+}:\mathbb{Q}]+d-2}|A_K|(1-j) \mod \mathbb{Z}_2^{*}.\ 
\end{equation}
where $d$ is the number of primes of $K^{+}$ which ramify in $K$.\ Let $\xi $ be the non-trivial character of $\Gal(K/K^{+})$.\ Then we have $\Ind^{H}_{\Gal(K/K^{+})} (\xi) = \sum_{m \in M} \phi^m$ and 
\[
 \xi(\theta_{K/K^{+}})=L(0,\xi,K/K^{+}) = \prod_{m \in M} L(0,\phi^{m},K/F).\ 
\] 
By (\ref{eq:sands}),\ we have 
\begin{eqnarray*}
 |A_K| &=& \omega_K \xi(\theta_{K/K^{+}}) 2^{-[K^{+}:\mathbb{Q}]-d+1} \\
        &=& \omega_K \prod_{m \in M} L(0,\phi^{m},K/F) 2^{-[K^{+}:\mathbb{Q}]-d+1} 
\end{eqnarray*}
where the equality is used in the sense that the $2$-parts of the both sides coincide.\ Since $\omega_{K,2} =2$ by Lemma \ref{lem:roots},\ we also have
\begin{equation*}
 |A_K| = \prod_{m \in M} L(0,\phi^{m},K/F) 2^{-[K^{+}:\mathbb{Q}]-d+2} 
\end{equation*}
Since $[K^{+}:\mathbb{Q}] \geq 2^{n+1}$ (recalling that $K/F$ is contained in some $Q_{2^{n+2}}$-extension),\ we get $-[K^{+}:\mathbb{Q}]-d+2 \leq -2^{n+1}+2$.\ Hence,\ we also get 
\begin{eqnarray}\label{eq:classeq3}
 |A_K| &\leq & \prod_{m \in M} L(0,\phi^{m},K/F) 2^{-2^{n+1}+2} \nonumber \\ 
        &=     &  \frac{4}{2^{2^{n}}}  N_{\mathbb{Q}(\zeta_{2^{n+1}})/\mathbb{Q}}(\frac{L(0,\phi,K/F)}{2})\nonumber \\
        &\leq & N_{\mathbb{Q}(\zeta_{2^{n+1}})/\mathbb{Q}}(\frac{L(0,\phi,K/F)}{2}) 
\end{eqnarray}
and the right hand side of the last equality lies in $\mathbb{Z}_2$ and hence $(1/2)L(0,\phi,K/F)$ lies in $\mathbb{Z}_2[\zeta_{2^{n+1}}]$.\ 
Next we treat the module $(\sum_{m \in M^{+}} x_{m}\pr_{\chi_{m}} A_K)$ and regard this module as a $\mathbb{Z}_2[\zeta_{2^{n+1}}]$-module.\ Then by (\ref{eq:classeq3}),\ we have 
\[
|(\sum_{m \in M^{+}}x_{m} \pr_{\chi_{m}} A_K)| \leq [\mathbb{Z}_2[\zeta_{2^{n+1}}]:((1/2)L(0,\phi,K/F))] =  [\mathbb{Z}_2[\zeta_{2^{n+1}}]:((1/2)\overline{\theta}_{K/F})]
\] 
where $\overline{\theta}_{K/F}$ is the image of $\theta_{K/F}$ under the surjection $\mathbb{Z}_2[H] \twoheadrightarrow \mathbb{Z}_2[\zeta_{2^{n+1}}]$.\ This implies $(1/2)(\sum_{m \in M^{+}}x_{m} \pr_{\chi_{m}} )\theta_{K/F}$ annihilates $A_K$.\ Then for any fractional ideal $\mathfrak{A}$ of $K$ whose class in $Cl(K)$ is of $2$-power order,\ we have that $\mathfrak{A}^{(1/2)(\sum_{m \in M^{+}} x_{m}\pr_{\chi_{m}} )\theta_{K/F}}$ lies in $P_K \cdot I_K^{+}$ where $P_K$ is the group of principal ideals of $K$ and hence we have $\mathfrak{A}^{(1/2)(\sum_{m \in M^{+}} x_{m}\pr_{\chi_{m}})\theta_{K/F}(1-j)}=\mathfrak{A}^{\sum_{m \in M} \pr_{\phi^{m}} \theta_{K/F}}$ lies in $P_K^{1-j}$.\ This completes the proof.\ 
\end{proof}

\subsection{CM-extensions with group $\mathbb{Z}/2\mathbb{Z} \times A_4$}\label{sec:alt}
Let $K/k$ be a finite Galois extension whose Galois group is isomorphic to $\mathbb{Z}/2\mathbb{Z} \times A_4$ where $A_4$ is the alternating group on $4$ letters.\ we regard $A_4$ as the group of even permutation of the set $\{1,\ 2,\ 3,\ 4\}$.\ Since the center of $A_4$ is trivial,\ the generator of $\mathbb{Z}/2\mathbb{Z}$ corresponds to the unique complex conjugation $j$.\
\subsubsection{Characters of $\mathbb{Z}/2\mathbb{Z} \times A_4$}
We set $x=(12)(34)$ and $y=(123)$
The irreducible characters of $\mathbb{Z}/2\mathbb{Z} \times A_4$ are determined by the following character table,\ where $\{ \cdot \}$ indicates conjugacy classes:
\begin{table}[H]
\caption{The character table of $\mathbb{Z} /2\mathbb{Z} \times A_4$}
\begin{center}
{\begin{tabular}{|c||c|c|c|c|c|c|c|c|}
\hline
             &$\{1\}$ & $\{ x \}$ & $\{ yx \}$ & $\{ y^{2}x\}$ & $\{ j \}$ & $\{ jx \}$ & $\{ jyx\}$ & $\{ j y^{2} x\}$ \\
\hline
$\chi_1$ & $1$ & $1$ & $1$ & $1$ & $1$ & $1$ & $1$ & $1$\\
\hline
$\chi_2$ & $1$ & $1$ & $1$ & $1$ & $-1$ & $-1$ & $-1$ & $-1$\\
\hline
$\chi_3$ & $1$ & $1$ & $\zeta_3$ & $\zeta_3^{2}$ & $1$ & $-1$ & $\zeta_3$ & $\zeta_3^{2}$ \\
\hline
$\chi_4$ & $1$ & $1$ & $\zeta_3$ & $\zeta_3^{2}$ & $-1$ & $-1$ & $-\zeta_3$ & -$\zeta_3^{2}$ \\
\hline
$\chi_5$ & $1$ & $1$ & $\zeta_3^{2}$ & $\zeta_3$ & $1$ & $1$ &$\zeta_3^{2}$ & $\zeta_3$\\
\hline
$\chi_6$ & $1$ & $1$ & $\zeta_3^{2}$ & $\zeta_3$ & $-1$ & $-1$ &$-\zeta_3^{2}$ & $-\zeta_3$\\
\hline
$\chi_7$ & $3$ & $-1$ & $0$ & $0$ & $3$ & $-1$ & $0$ & $0$\\
\hline
$\chi_8$ & $3$ & $-1$ & $0$ & $0$ & $-3$ & $1$ & $0$ & $0$\\
\hline
\end{tabular}}
\end{center}
\end{table} 
\noindent
From the above table,\ we see that the only odd characters are $\chi_2,\ \chi_4,\ \chi_6$ and $\chi_8$.\ Since $\ker \chi_2$ has index $2$,\ the corresponding subextension $K_2/k$ is a quadratic extension,\ and since we have $\ker \chi_4 =\ker \chi_6$ and  this subgroup has index $6$,\ we have $K_4 = K_6$ and $K_4/k$ is a cyclic extension of degree $6$.\ Let $V$ be Klein subgroup of $A_4$ and $\phi_{6,1},\ \phi_{6,2}$ and $\phi_{6,3}$ be characters of $\mathbb{Z}/2\mathbb{Z} \times V$ whose restriction to $V$ are non-trivial.\ Then we have $\Ind^{\mathbb{Z}/2\mathbb{Z} \times A_k}_{\mathbb{Z}/2\mathbb{Z} \times V}(\phi_{6,i}) = \chi_6$ for $i=1,\ 2,\ 3$ and the indices of their kernel in $\mathbb{Z}/2\mathbb{Z} \times V$ is $2$.\ Hence we see that $k_6 = K^{\mathbb{Z}/2\mathbb{Z} \times V}$ and $K_{6,i}/k_6$ is a quadratic extension for all $i$.\ 
\subsubsection{Proof of conjectures for extensions with group $\mathbb{Z}/2\mathbb{Z} \times A_4$}
In this subsection,\ we prove the following theorem by using Theorem \ref{thm:brumer-stark}.\ 
\begin{thm}\label{thm:brumer3}
Let $K/k$ be a finite Galois CM-extension whose Galois group is isomorphic to $\mathbb{Z}/2\mathbb{Z} \times A_4$ and $S$ be a finite set of places of $k$ which contains all infinite places.\ Then 
\begin{description}
\item[(1)] the $2$-part and the $3$-part of the weak non-abelian Brumer conjecture and the weak non-abelian Brumer-Stark conjecture are true for $K/k$ and $S$,\ 
\item[(2)] for each odd prime $l$ apart from $3$ which does not split in $\mathbb{Q}(\zeta_{3})$,\  the $l$-part of the non-abelian Brumer conjecture and the non-abelian Brumer-Stark conjecture are true for $K/k$ and $S$.\ 
\end{description}
\end{thm}
\begin{rem}
{\rm (1)} In the case $k = \mathbb{Q}$,\ the above results except the $2$-part is contained in Nickel's work \cite{Ni11},\ \cite{Nib} if we assume $\mu=0$ as well as Theorem \ref{thm:brumer1}.\ \\
{\rm (2)} If no prime above $p$ splits in $K/K^{+}$ whenever $K^{cl} \subset (K^{cl})^{+}(\zeta_p)$,\ the above result holds for odd $p$ by \cite[Corollary  4.2]{Nia}.\ \\ 
\end{rem}
The observation in the previous subsection tells us that we have only to verify the Brumer-Stark conjecture for two relative quadratic extensions $K_2/k$,\ $K_4/k$ and $K_2^{\prime}/k_2$.\ By \cite{Ta}[\S3,\ case(c)],\ the Brumer-Stark conjecture is true for any relative quadratic extensions and hence true for $K_2/k,\ K_{6,1}/k_6,\ K_{6,2}/k_6$ and $K_{6,3}/k_6$.\ In order to complete the proof of Theorem \ref{thm:brumer3},\  we have to verify the $l$-part of the Brumer-Stark conjecture for $K_4/k$ for each prime $l$ which does not split in $\mathbb{Q}(\zeta_3)$.\ However,\ the proof of Theorem \ref{thm:brumer-stark} (and Lemma \ref{lem:relation3}) tells us that we only have to verify the slightly weaker annihilation result,\ that is,\ we  only need  (\ref{eq:weakann3}) and  (\ref{eq:weakann4}) for $K_4/k$.\ To do that,\ it is enough to prove the following proposition: 
\begin{prop}\label{cl:6p2}
Let $l$ be a prime which does not split in $\mathbb{Q}(\zeta_3)$.\ Let $F/k$ be any cyclic CM-extension of number fields of degree $6$.\ We assume $K$ contains $F$ so that $K/k$ is CM with Galois group $\mathbb{Z}/2\mathbb{Z} \times A_4$.\ We let $\sigma$ be a generator of the Galois group of $F^{\langle j \rangle}/k$ (hence $\Gal(F/k)=\langle \sigma j \rangle$).\  Take any element of the form $x_{\chi_6} \pr_{\chi_6}$ in $\mathfrak{F}(\mathbb{Z}/2\mathbb{Z}\times A_4)$.\ Then for any fractional ideal $\mathfrak{A}$ of $F$ whose class in $Cl(F)$ is of $l$-power order,\ 
\begin{description}
 \item[(1)] $\mathfrak{A}^{\omega_F x_{\chi_6} \pr_{\chi_6}\theta_{F/k}} = (\alpha)$ for some anti-unit $\alpha \in F^{*}$,\ 
 \item[(2)] $F(\alpha^{1/\omega_{F,l}})/k$ is abelian.\
\end{description} 
where $\omega_{F,l}$ is the $l$-part of $\omega_F$.\ 
\end{prop}

\begin{proof}
Exactly the same proof as Proposition \ref{cl:6p} works since the only fact we need is that $(1 + \sigma + \sigma^{2}) x_{\chi_6}\pr_{\chi_6} A_F = 0$.\ 
\end{proof}
%
  
Deoartment of Mathematics,\ Keio University, \\
Yokohama, 223-8522, Japan\\
 e-mail: jiro-math@z2.keio.jp
\end{document}